\documentclass[a4paper]{article}

\usepackage{amsmath}
\usepackage{amssymb}
\usepackage{amsthm}
\usepackage{bm}
\usepackage{booktabs}
\usepackage[hang,flushmargin]{footmisc} 
\usepackage[margin=25mm]{geometry}
\usepackage[UKenglish]{isodate}
	\cleanlookdateon
\usepackage{mathpazo}
\usepackage{microtype}
\usepackage{sectsty}
	\sectionfont{\large}
	\subsectionfont{\normalsize}
\usepackage{titlesec}
	\titlespacing{\section}{0pt}{12pt}{0pt}
	\titlespacing{\subsection}{0pt}{6pt}{0pt}
\usepackage[nottoc]{tocbibind}
\usepackage{xcolor}
	\definecolor{linkred}{rgb}{0.6,0,0}
	\definecolor{linkblue}{rgb}{0,0,0.6}
\usepackage[
	pdftitle={Towards the topological recursion for double Hurwitz numbers},
	pdfauthor={Norman Do and Maksim Karev},
	ocgcolorlinks,
	linkcolor=linkblue,
	citecolor=linkred,
	urlcolor=linkred]
{hyperref}

\theoremstyle{plain}
	\newtheorem{theorem}{Theorem}
	\newtheorem{proposition}[theorem]{Proposition}
	\newtheorem{corollary}[theorem]{Corollary}
	\newtheorem{conjecture}[theorem]{Conjecture}
	\newtheorem{lemma}[theorem]{Lemma}

\theoremstyle{definition}
	
	\newtheorem{definition}[theorem]{Definition}
	\newtheorem{remark}[theorem]{Remark}

\newcommand\blfootnote[1]{
	\begingroup
	\renewcommand\thefootnote{}\footnote{#1}
	\addtocounter{footnote}{-1}
	\endgroup
}

\newcommand {\dd}{\mathrm{d}}
\newcommand {\h}{\hbar}
\newcommand{\llambda}{\boldsymbol{\lambda}}
\newcommand{\mmu}{\boldsymbol{\mu}}

\newcommand {\xx}{\mathbf{x}}

\newcommand {\zz}{\mathbf{z}}

\begin{document}

\enlargethispage{2.5pt}

{\large \bfseries Towards the topological recursion for double Hurwitz numbers}

{\bfseries Norman Do and Maksim Karev}

Single Hurwitz numbers enumerate branched covers of the Riemann sphere with specified genus, prescribed ramification over infinity, and simple branching elsewhere. They exhibit a remarkably rich structure. In particular, they arise as intersection numbers on moduli spaces of curves and are governed by the topological recursion of Chekhov, Eynard and Orantin. Double Hurwitz numbers are defined analogously, but with prescribed ramification over both zero and infinity. Goulden, Jackson and Vakil have conjectured that double Hurwitz numbers also arise as intersection numbers on moduli spaces.

In this paper, we repackage double Hurwitz numbers as enumerations of branched covers weighted by certain monomials and conjecture that they are governed by the topological recursion. Evidence is provided in the form of the associated quantum curve and low genus calculations. We furthermore reduce the conjecture to a weaker one, concerning a certain polynomial structure of double Hurwitz numbers. Via the topological recursion framework, our main conjecture should lead to a direct connection to enumerative geometry, thus shedding light on the aforementioned conjecture of Goulden, Jackson and Vakil.
\blfootnote{
\emph{Date:} \today \\ 
The first author was supported by the Australian Research Council grant DE130100650.} 

~

\hrule

\setlength{\parskip}{0pt}
\tableofcontents
\setlength{\parskip}{4pt}

\section{Introduction} \label{sec:introduction}

\subsection{Double Hurwitz numbers}

Single Hurwitz numbers --- also known as simple Hurwitz numbers --- enumerate branched covers of the Riemann sphere with specified genus, prescribed ramification over infinity, and simple branching elsewhere. Although first studied by Hurwitz towards the end of the nineteenth century~\cite{hur}, the extent of their remarkably rich structure only became apparent towards the end of the twentieth. The revival of interest in Hurwitz numbers was in part sparked by the empirical observation of Goulden, Jackson and Vainshtein that they possess a certain polynomial structure~\cite{gou-jac-vai}. This was later proved by Ekedahl, Lando, Shapiro and Vainshtein, who showed that single Hurwitz numbers are equal to Hodge integrals over the Deligne--Mumford compactification of the moduli space of curves~\cite{eke-lan-sha-vai}. Their so-called ELSV formula not only makes the polynomial structure of Hurwitz numbers apparent, but connects them to the realms of enumerative geometry and mathematical physics. More recently, work motivated by topological string theory led Bouchard and Mari\~{n}o to conjecture that single Hurwitz numbers are governed by the topological recursion of Chekhov, Eynard and Orantin~\cite{bou-mar}. This was subsequently proven~\cite{eyn-mul-saf} and it has furthermore been demonstrated that the ELSV formula and the Bouchard--Mari\~{n}o conjecture are in some sense equivalent~\cite{dun-kaz-ora-sha-spi}.

Double Hurwitz numbers enumerate branched covers of the Riemann sphere with specified genus, prescribed ramification over both zero and infinity, and simple branching elsewhere. Okounkov showed that they arise naturally as coefficients of a certain tau-function of the Toda integrable hierarchy~\cite{oko}. Goulden, Jackson and Vakil demonstrated that double Hurwitz numbers possess a certain piecewise polynomial structure. Moreover, they presented evidence to suggest that double Hurwitz numbers are equal to integrals over moduli spaces of curves equipped with a line bundle~\cite{gou-jac-vak}. However, to this date, a rigorous definition of these moduli spaces and their intersection theory is yet to be determined.

In this paper, we interpret double Hurwitz numbers as enumerations of branched covers weighted by monomials in the following way.

\begin{definition} \label{def:dhurwitz}
Fix a positive integer $d$ and weights $s, q_1, q_2, \ldots, q_d \in \mathbb{C}$. Define the \emph{double Hurwitz number} $DH_{g,n}(\mu_1, \ldots, \mu_n) $ to be the weighted count of connected genus $g$ branched covers of the Riemann sphere $f: (\Sigma; p_1, \ldots, p_n) \to (\mathbb{CP}^1; \infty)$ such that
\begin{itemize}
\item all branching away from 0 and $\infty$ is simple and occurs at some number $m$ of fixed points;
\item $f^{-1}(\infty) = \mu_1 p_1 + \cdots + \mu_n p_n$; and
\item no preimage of 0 has ramification index larger than $d$.
\end{itemize}
If such a branched cover has ramification profile $(\lambda_1, \lambda_2, \ldots, \lambda_\ell)$ over 0, then we assign it the weight
\[
\frac{q_{\lambda_1} q_{\lambda_2} \cdots q_{\lambda_{\ell}}}{|\mathrm{Aut}~ f|} \frac{s^m}{m!}.
\]
Here, the automorphism group $\mathrm{Aut}~ f$ consists of Riemann surface automorphisms $\phi: \Sigma \to \Sigma$ that preserve the marked points $p_1, \ldots, p_n$ and satisfy $f \circ \phi = f$.
\end{definition}

\begin{remark}
We make note of the following initial remarks concerning Definition~\ref{def:dhurwitz}.
\begin{itemize}
\item The Riemann--Hurwitz formula asserts that the number of simple branch points must satisfy $m = 2g - 2 + n + \ell$, where $\ell$ is the number of preimages of 0.

\item Each double Hurwitz number is a weighted homogeneous polynomial in $q_1, q_2, \ldots, q_d$ with positive rational coefficients.

\item The double Hurwitz numbers $H^g_{\llambda, \mmu}$ defined by Goulden, Jackson and Vakil~\cite{gou-jac-vak} may be recovered as a combinatorial factor multiplied by the coefficient of $q_{\lambda_1} q_{\lambda_2} \cdots q_{\lambda_\ell}$ in the polynomial $DH_{g,n}(\mu_1, \ldots, \mu_n)$, as long as we fix $d \geq \max(\lambda_1, \lambda_2, \ldots, \lambda_\ell)$.

\item The parameter $s$ is redundant in the sense that its exponent in each monomial of a double Hurwitz number can be recovered via the Riemann--Hurwitz formula. However, the advantages of retaining it should become apparent in the following.

\item The single Hurwitz numbers are recovered by taking $q_1 = 1$ and $q_i = 0$ for $i \neq 1$. More generally, one recovers $a$-orbifold Hurwitz numbers by taking $q_a = 1$ and $q_i = 0$ for $i \neq a$~\cite{do-lei-nor, bou-her-liu-mul}.
\end{itemize}
\end{remark}

See Appendix~\ref{sec:data} for a table of double Hurwitz numbers.

\subsection{The main conjecture}

The particular way in which we have assembled the double Hurwitz numbers allows us to consider potential analogues of results pertaining to single Hurwitz numbers. As an example, we propose a vast generalisation of the Bouchard--Mari\~{n}o conjecture, namely that the double Hurwitz numbers defined above are governed by the topological recursion.

The topological recursion of Chekhov, Eynard and Orantin arose from the theory of matrix models and has subsequently found widespread applications to various areas of mathematics and physics~\cite{che-eyn,eyn-ora07a}. Beyond the realm of matrix models, it is now either known or conjectured to govern the following problems: the enumeration of ribbon graphs and hypermaps~\cite{dum-mul-saf-sor, do-man, dun-ora-pop-sha}; Hurwitz numbers of various flavours~\cite{bou-mar, eyn-mul-saf, do-lei-nor, bou-her-liu-mul, do-dye-mat, do-kar}; Gromov--Witten invariants of $\mathbb{CP}^1$~\cite{nor-sco, dun-ora-sha-spi}; Gromov--Witten invariants of toric Calabi--Yau threefolds~\cite{bou-kle-mar-pas, eyn-ora12, fan-liu-zon}; and asymptotics of coloured Jones polynomials of knots~\cite{dij-fuj-man, bor-eyn}.

Let us first state our main conjecture relating double Hurwitz numbers to the topological recursion before presenting a rigorous treatment of the topological recursion itself.

\begin{conjecture} \label{con:main}
Let $P(z) = q_1z + q_2z^2 + \cdots + q_d z^d$. Topological recursion applied to the rational spectral curve
\begin{equation} \label{eq:scurve}
x(z) = z \exp (-sP(z)) \qquad \text{and} \qquad y(z) = P(z)
\end{equation}
produces correlation differentials whose expansions at $x_i = 0$ satisfy
\[
\omega_{g,n} = \sum_{\mu_1, \ldots, \mu_n=1}^\infty DH_{g,n}(\mu_1, \ldots, \mu_n)  \prod_{i=1}^n \mu_i x_i^{\mu_i-1} \dd x_i, \qquad \text{for } (g,n) \neq (0,2).
\]
\end{conjecture}

In general, the topological recursion takes as input the data of a spectral curve and outputs correlation differentials $\omega_{g,n}$ for integers $g \geq 0$ and $n \geq 1$. We now describe explicitly the topological recursion in the context of Conjecture~\ref{con:main}; for an exposition of the topological recursion in greater generality, the reader may consult the literature~\cite{eyn-ora07a}.

First, let $a_1, a_2, \ldots, a_d$ be the branch points of the spectral curve --- in other words, the zeroes of
\[
\dd x(z) = 0 \qquad \Leftrightarrow \qquad s z P'(z) - 1 = 0.
\]
For the following discussion, we will require the mild assumption that these zeroes are distinct and hence simple.\footnote{One can also state the conjecture in the case of higher order zeroes by invoking the global topological recursion of Bouchard and Eynard~\cite{bou-eyn}. We consider only the generic case in order to streamline the presentation. Indeed, Bouchard and Eynard demonstrate that non-simple spectral curves and their correlation differentials can be obtained in the limit of simple spectral curves.} It follows that there exists a local involution $\sigma_i: U_i \to U_i$ defined on a small open neighbourhood of the branch point $z = a_i$ on the spectral curve, such that $\sigma_i$ is meromorphic and satisfies the equation $x(\sigma_i(z)) = x(z)$, but is not the identity on $U_i$.

Define the base cases\footnote{In the original formulation of the topological recursion, one usually defines $\omega_{0,1}(z_1) = - y(z_1) \, \dd x(z_1)$. The modification here assumes the opposing sign convention and uses $y(z_1) \, \dd \log x(z_1)$ instead, which applies in various settings, such as the Bouchard--Mari\~{n}o conjecture concerning single Hurwitz numbers. In such cases, it is common in the literature to refer to $x$ as a $\mathbb{C}^*$-coordinate, rather than a $\mathbb{C}$-coordinate. Thus, equation~\eqref{eq:scurve} describes what is known as a $\mathbb{C}^* \times \mathbb{C}$ spectral curve.}
\begin{equation} \label{eq:trbases}
\omega_{0,1}(z_1) = \frac{y(z_1) \, \dd x(z_1)}{x(z_1)} \qquad \text{and} \qquad \omega_{0,2}(z_1, z_2) = \frac{\dd z_1 \otimes \dd z_2}{(z_1-z_2)^2}.
\end{equation}

Next, define the recursion kernel
\begin{equation} \label{eq:kernel}
K_i(z_1, z) = \frac{\int_o^z \omega_{0,2}(z_1, \,\cdot\,)}{\omega_{0,1}(z) - \omega_{0,1}(\sigma_i(z))} = \frac{\dd z_1}{z_1 - z} \frac{1}{\omega_{0,1}(z) - \omega_{0,1}(\sigma_i(z))},
\end{equation}
which exists only on the small open neighbourhood $U_i$ of the branch point $z = a_i$. The topological recursion is not sensitive to the choice of basepoint $o$ on the spectral curve, so we have taken it to be $z = \infty$ for convenience.

For $2g-2+n > 0$, let
\[
\omega_{g,n}(z_1, \ldots, z_n) = \sum_{i=1}^d \mathop{\text{Res}}_{z=a_i} K_i(z_1,z) \Bigg[ \omega_{g-1,n+1}(z, \sigma_i(z), \zz_S) + \mathop{\sum_{g_1+g_2=g}}_{I \sqcup J = S \setminus \{1\}}^\circ \omega_{g_1,|I|+1}(z, \zz_I) \, \omega_{g_2,|J|+1}(\sigma_i(z), \zz_J) \Bigg].
\]
Here, we let $S = \{1, 2, \ldots, n\}$ and let $\zz_I = \{z_{i_1}, z_{i_2}, \ldots, z_{i_k}\}$ for $I = \{i_1, i_2, \ldots, i_k\}$. The symbol $\circ$ over the inner summation means that we exclude terms involving $\omega_{0,1}$.

The correlation differential $\omega_{g,n}$ is a multidifferential on the spectral curve $\mathcal{C}$ rather than a differential form. More precisely, $\omega_{g,n}$ is a meromorphic section of the line bundle $\pi_1^*(T^*\mathcal{C}) \otimes \pi_2^*(T^*\mathcal{C}) \otimes \cdots \otimes \pi_n^*(T^*\mathcal{C})$, on the Cartesian product $\mathcal{C}^n$, where $\pi_i: \mathcal{C}^n \to \mathcal{C}$ denotes projection onto the $i$th factor. For notational convenience, we will subsequently drop the $\otimes$ symbol when writing multidifferentials. Despite the fact that the topological recursion is asymmetric in nature, the resulting correlation differentials are indeed symmetric. Furthermore, for $2g-2+n>0$, it is known that the correlation differential $\omega_{g,n}(z_1, \ldots, z_n)$ has poles only at the branch points $a_1, a_2, \ldots, a_d$ of the spectral scurve.

Finally, to interpret the statement of Conjecture~\ref{con:main}, one is required to expand the correlation differentials $\omega_{g,n}(z_1, \ldots, z_n)$ at $x_i = 0$. We do this by setting $x_i = x(z_i)$. This notation and the analogous notation $y_i = y(z_i)$ will be used throughout the paper.

\subsection{Evidence for the conjecture}

There is substantial evidence to support our main conjecture. The cut-and-join recursion, stated below as Proposition~\ref{prop:cutjoin}, allows one to compute double Hurwitz numbers recursively. It is natural to present this at the level of generating functions $F_{g,n}$ known as free energies, which store all double Hurwitz numbers of the form $DH_{g,n}(\mu_1, \ldots, \mu_n)$. In particular, we obtain a relation between the free energies, stated below as Corollary~\ref{cor:cutjoin2}, that superficially resembles the topological recursion. In particular, it expresses $F_{g,n}$ in terms of $F_{g-1,n+1}$ and products $F_{g_1,n_1} \times F_{g_2, n_2}$ that do not involve $F_{0,1}$, where $g_1+g_2=g$ and $n_1+n_2=n+1$.

In many instances of the topological recursion on a spectral curve, there is an associated quantum curve that underlies it. In short, a quantum curve is a differential operator that annihilates a wave function constructed from the correlation differentials of the topological recursion. One can consider the semi-classical limit of the quantum curve, which is a plane curve that coincides with the original spectral curve in many cases. The quantum curve for double Hurwitz numbers was previously computed by Alexandrov, Lewanski and Shadrin~\cite{ale-lew-sha} and its semi-classical limit is $y = P(x \exp(sy))$, which recovers the spectral curve of equation~\eqref{eq:scurve}. It has been proposed that the existence of a quantum curve can be used to predict the structure of topological recursion and the associated spectral curve~\cite{nor16}.

Finally, one can attempt to verify the main conjecture via direct computation in low genus. We carry this out in the cases $(g,n) = (0,1)$, $(0,3)$ and $(1,1)$.

\subsection{Ramifications and applications}

Our main conjecture subsumes some existing results and conjectures in the literature, which can be obtained by specialising the weights $s, q_1, q_2, \ldots$. It is inspired by the Bouchard--Mari\~{n}o conjecture, which is recovered by considering the case $P(z) = z$. More generally, one can consider $P(z) = z^a$ for a positive integer $a$ and recover the fact that the topological recursion governs orbifold Hurwitz numbers~\cite{do-lei-nor, bou-her-liu-mul}. Alexandrov, Lewanski and Shadrin propose a 1-parameter deformation of single Hurwitz numbers, via the enumeration of double Hurwitz numbers, weighted by $c$ to the power of the colength of the ramification profile over 0~\cite{ale-lew-sha}. Our main conjecture recovers theirs by specialising the weights to $q_i = i c^{i-1}$ and sending $d$ to infinity or, equivalently, by considering the case $P(z) = \frac{z}{(1-cz)^2}$.

Our strategy of repackaging double Hurwitz numbers as enumerations of branched covers weighted by monomials allows us to generalise results concerning single Hurwitz numbers to their double counterparts. This approach may also be employed in the context of other problems of a similar nature, such as double montone Hurwitz numbers, though we do not pursue that line of reasoning here~\cite{do-dye-mat, do-kar}.

Our main conjecture, stated above as Conjecture~\ref{con:main}, leads to certain previously unidentified structure for double Hurwitz numbers. For instance, it implies that the free energy generating functions for double Hurwitz numbers defined below in equation~\eqref{eq:fenergies} are actually rational, with poles only at the branch points of the spectral curve satisfying a certain symmetry with respect to the local involution. We state this conjecture explicitly below as Conjecture~\ref{con:poly}. At the level of the double Hurwitz numbers themselves, this manifests as a particular polynomial-like structure, which we state explicitly below as Conjecture~\ref{con:elsv}. Perhaps the main consequence of this paper is that these three conjectures are all equivalent.

We foresee that the main application of our conjecture is towards the geometry of double Hurwitz numbers. The general theory of topological recursion has been motivated in part by its connections to enumerative geometry. A rather simple statement is that spectral curves are locally modelled on the Airy curve $x(z) = \frac{1}{2} z^2$ and $y(z) = z$ at their branch points. A deeper statement is that correlation differentials of arbitrary spectral curves are also locally modelled on those of the Airy curve. Furthermore, the Airy correlation differentials are known to store psi-class intersection numbers on the moduli space of curves, the main objects of the celebrated Witten--Kontsevich theorem~\cite{eyn-ora07a}. Eynard managed to push this result further and show that the lower order terms of correlation differentials can also be related to intersection numbers on moduli spaces~\cite{eyn11b}. More recently, Dunin-Barkowski, Orantin, Shadrin and Spitz made a connection between topological recursion and the Givental formalism, proving that correlation differentials store ancestor invariants of a cohomological field theory~\cite{dun-ora-sha-spi}.

From the previous discussion, a consequence of our main conjecture is a direct relation between double Hurwitz numbers and intersection theory on moduli spaces of curves. Such a connection should shed light on the conjecture of Goulden, Jackson and Vakil, which asserts that double Hurwitz numbers arise as intersection numbers on certain moduli spaces of curves equipped with a line bundle~\cite{gou-jac-vak}.

\begin{itemize}
\item In Section~\ref{sec:combinatorics}, we state the cut-and-join recursion for double Hurwitz numbers, which previously appeared in the work of Zhu~\cite{zhu12}. We write these at the level of the generating functions known as free energies. Finally, we discuss the notion of pruned double Hurwitz numbers, which previously appeared in the work of Hahn~\cite{hah}, and their relation to the free energies.

\item In Section~\ref{sec:evidence}, we give evidence to support our main conjecture. First, we present the quantum curve for double Hurwitz numbers, which was first deduced by Alexandrov, Lewanski and Shadrin~\cite{ale-lew-sha}, and show that its semi-classical limit does indeed recover the spectral curve of equation~\eqref{eq:scurve}. We then provide low genus evidence by calculating the free energies $F_{0,3}$ and $F_{1,1}$ and demonstrating that they are consistent with our main conjecture.

\item In Section~\ref{sec:proof}, we outline a possible proof of our main conjecture. In particular, we show that the conjecture can be reduced to proving that the free energies satisfy so-called linear loop equations. We furthermore show that these constraints are equivalent to a polynomial-like structure for double Hurwitz numbers.
\end{itemize}

\section{Combinatorics of double Hurwitz numbers} \label{sec:combinatorics}

\subsection{Cut-and-join recursion} \label{subsec:cutjoin}

A natural way to compute Hurwitz numbers is via the cut-and-join recursion. It was originally formulated by Goulden and Jackson in the case of genus 0 single Hurwitz numbers~\cite{gou-jac}, but has much broader applicability. At the level of branched covers, the cut-and-join recursion arises by examining the behaviour of the ramification profile over infinity as one of the simple branch points approaches infinity. Otherwise, one may interpret Hurwitz numbers as an enumeration of transitive factorisations
\begin{equation} \label{eq:factorisation}
\tau_0 \tau_1 \tau_2 \cdots \tau_m = \rho
\end{equation}
in symmetric groups, by passing to the monodromy representation of a branched cover and appealing to the Riemann existence theorem. Here, $\tau_0$ represents the monodromy over 0, the transpositions $\tau_1, \tau_2, \ldots, \tau_m$ represent the monodromy of the simple branch points, and $\rho$ represents the inverse of the monodromy over $\infty$. The cut-and-join recursion then arises simply by considering the result of multiplying both sides of equation~\eqref{eq:factorisation} on the right by the transposition $\tau_m$.

The cut-and-join recursion for double Hurwitz numbers appears in the work of Zhu~\cite{zhu12}. The statement below paraphrases the result using our particular definition of double Hurwitz numbers.

\begin{proposition}[Cut-and-join recursion] \label{prop:cutjoin}
The double Hurwitz numbers satisfy the equation
\begin{align*}
&\bigg( 2g-2+n + \sum_{i=1}^d q_i \frac{\partial}{\partial q_i} \bigg) DH_{g,n}(\mu_1, \ldots, \mu_n) = s \sum_{i < j} (\mu_i + \mu_j) \, DH_{g,n-1}(\mmu_{S \setminus \{i,j\}}, \mu_i+\mu_j) \\
&\qquad + \frac{s}{2} \sum_{i=1}^n \sum_{\alpha + \beta = \mu_i} \alpha \beta \bigg[ DH_{g-1,n+1}(\alpha, \beta, \mmu_{S \setminus \{i\}}) + \mathop{\sum_{g_1+g_2=g}}_{I \sqcup J = S \setminus \{i\}} DH_{g_1,|I|+1}(\alpha, \mmu_I) \, DH_{g_2, |J|+1}(\beta, \mmu_J) \bigg].
\end{align*}
Here, we let $S = \{1, 2, \ldots, n\}$ and let $\mmu_I = \{\mu_{i_1}, \mu_{i_2}, \ldots, \mu_{i_k}\}$ for $I = \{i_1, i_2, \ldots, i_k\}$. Furthermore, all double Hurwitz numbers can be calculated from this recursion along with the base cases
\[
\left. DH_{0,1}(\mu) \right|_{s=0} = \frac{1}{\mu} q_{\mu} \qquad \text{and} \qquad \left. DH_{g,n}(\mu_1, \ldots, \mu_n) \right|_{s=0} = 0 \quad \text{for } (g,n) \neq (0,1).
\]
\end{proposition}

Recall the introduction of the parameter $s$, which records the number of simple branch points, in our definition of the double Hurwitz numbers. One advantage of retaining the parameter is the following slightly more compact form of the cut-and-join recursion.

\begin{corollary} \label{cor:cutjoin}
The double Hurwitz numbers satisfy the equation
\begin{align*}
&\frac{\partial}{\partial s} DH_{g,n}(\mu_1, \ldots, \mu_n) = \sum_{i < j} (\mu_i + \mu_j) \, DH_{g,n-1}(\mmu_{S \setminus \{i,j\}}, \mu_i+\mu_j) \\
& \qquad+ \frac{1}{2} \sum_{i=1}^n \sum_{\alpha + \beta = \mu_i} \alpha \beta \bigg[ DH_{g-1,n+1}(\alpha, \beta, \mmu_{S \setminus \{i\}}) + \mathop{\sum_{g_1+g_2=g}}_{I \sqcup J = S \setminus \{i\}} DH_{g_1,|I|+1}(\alpha, \mmu_I) \, DH_{g_2, |J|+1}(\beta, \mmu_J) \bigg].
\end{align*}
\end{corollary}

\begin{remark}
Observe that the mechanism behind the cut-and-join recursion is in some sense local. In the interpretation of Hurwitz numbers as enumerations of transitive factorisations, the recursion is not sensitive to the permutation $\tau_0$ in equation~\eqref{eq:factorisation}. So using the parameter $s$ allows us to express the cut-and-join recursion identically in the case of simple Hurwitz numbers~\cite{gou-jac}, orbifold Hurwitz numbers~\cite{do-lei-nor, bou-her-liu-mul}, as well as double Hurwitz numbers~\cite{zhu12}. In fact, we remark that the local nature of the cut-and-join recursion implies that it is not even sensitive to the topology of the base curve. Thus, it can also be expressed in the same way for Hurwitz numbers on base curves of higher genus, though with a different set of base cases~\cite{liu-mul-sor}.
\end{remark}

\subsection{Generating functions} \label{subsec:gfunctions}

It is natural to define the following generating functions for the double Hurwitz numbers. In the context of the topological recursion, such generating functions are commonly referred to as \emph{free energies}.
\begin{equation} \label{eq:fenergies}
F_{g,n}(x_1, \ldots, x_n) = \sum_{\mu_1, \ldots, \mu_n = 1}^\infty DH_{g,n}(\mu_1, \ldots, \mu_n) \prod_{i=1}^n x_i^{\mu_i}
\end{equation}
We also define the following formal multidifferentials, where $\dd_i$ denotes the exterior derivative with respect to the $i$th slot.
\begin{align*}
\Omega_{g,n}(x_1, \ldots, x_n) &= \dd_1 \cdots \dd_n F_{g,n}(x_1, \ldots, x_n) \\
&=\sum_{\mu_1, \ldots, \mu_n = 1}^\infty DH_{g,n}(\mu_1, \ldots, \mu_n) \prod_{i=1}^n \mu_i x_i^{\mu_i-1} \dd x_i
\end{align*}
Note that our main conjecture implies that when $2g-2+n>0$, these are in fact expansions of rational multidifferentials on the spectral curve of equation~\eqref{eq:scurve} and coincide with the correlation differentials generated by the topological recursion applied to the spectral curve.

\begin{proposition}\label{prop:cutandjoin1}
The free energies satisfy the equation
\begin{align*} 
\frac{\partial}{\partial s} F_{g,n}(x_1, \ldots, x_n) &= \sum_{i < j} \frac{x_ix_j}{x_i-x_j} \bigg[ \frac{\partial}{\partial x_i} F_{g,n-1}(\xx_{S \setminus \{j\}}) - \frac{\partial}{\partial_j} F_{g,n-1}(\xx_{S \setminus \{i\}}) \bigg] \\
&+ \frac{1}{2} \sum_{i=1}^n \bigg[ u_1 u_2 \frac{\partial^2}{\partial u_1 \, \partial u_2} F_{g-1,n+1}(u_1, u_2, \xx_{S \setminus \{i\}}) \bigg]_{\substack{u_1=x_i \\ u_2=x_i}} \\
&+ \frac{1}{2} \sum_{i=1}^n \mathop{\sum_{g_1+g_2=g}}_{I \sqcup J = S \setminus \{i\}} \bigg[ x_i \frac{\partial}{\partial x_i} F_{g_1,|I|+1}(x_i, \xx_I) \bigg] \bigg[ x_i \frac{\partial}{\partial x_i} F_{g_2,|J|+1}(x_i, \xx_J) \bigg].
\end{align*}
Here, we let $S = \{1, 2, \ldots, n\}$ and let $\xx_I = \{x_{i_1}, x_{i_2}, \ldots, x_{i_k}\}$ for $I = \{i_1, i_2, \ldots, i_k\}$. Furthermore, all free energies can be calculated from the initial conditions
\[
\left. F_{0,1}(x_1) \right|_{s=0} = \sum_{i=1}^d \frac{1}{i} q_i x_1^i \qquad \text{and} \qquad \left. F_{g,n}(x_1, \ldots, x_n) \right|_{s=0} = 0 \quad \text{for } (g,n) \neq (0,1).
\]
\end{proposition}

\begin{proof}
The strategy is to multiply both sides of the cut-and-join recursion of Corollary~\ref{cor:cutjoin} by $x_1^{\mu_1} \cdots x_n^{\mu_n}$ and to sum over all $\mu_1, \ldots, \mu_n$. The left side of the equation becomes
\begin{align*}
\sum_{\mu_1, \ldots, \mu_n=1}^\infty \frac{\partial}{\partial s} DH_{g,n}(\mu_1, \ldots, \mu_n) \prod_{i=1}^n x_i^{\mu_i} = \frac{\partial}{\partial s} F_{g,n}(x_1, \ldots, x_n).
\end{align*}

To calculate the first line on the right side of the equation, we use
\begin{align*}
\sum_{\mu_i, \mu_j=1}^\infty (\mu_i+\mu_j) \, DH(\mu_i+\mu_j) \, x_i^{\mu_i} x_j^{\mu_j} &= \sum_{\mu=1}^\infty \mu \, DH(\mu) \sum_{\alpha+\beta=\mu} x_i^{\alpha} x_j^{\beta} \\
&= \sum_{\mu=1}^\infty \mu \, DH(\mu) \, x_i x_j \frac{x_i^{\mu-1}-x_j^{\mu-1}}{x_i-x_j} = \frac{x_ix_j}{x_i-x_j} \bigg[ \frac{\partial}{\partial x_i} F(x_i) - \frac{\partial}{\partial x_j} F(x_j) \bigg].
\end{align*}

To calculate the second and third lines on the right side of the equation, we use
\[
\sum_{\mu=1}^\infty \sum_{\alpha+\beta=\mu} \alpha \beta \, DH(\alpha, \beta) \, x^{\mu} = \sum_{\alpha, \beta=1}^\infty \alpha \beta \, DH(\alpha, \beta) \, x^{\alpha + \beta} = \bigg[ u_1 u_2 \frac{\partial^2}{\partial u_1 \partial u_2} \ F(u_1, u_2) \bigg]_{\substack{u_1=x \\ u_2=x}}
\]
Note that we have dropped extraneous subscripts in the previous two equations. Finally, the initial conditions precisely capture the base cases of Proposition~\ref{prop:cutjoin}.
\end{proof}

In this particular form, the cut-and-join recursion can be used to determine the free energies uniquely. For instance, let us apply this to the case $(g,n) = (0,1)$.

It will be useful to consider the change of variables
\begin{equation} \label{eq:variables}
x(z,s') = z \exp \left( -s'P(z) \right) \qquad \text{and} \qquad s(z,s') = s',
\end{equation}
which we will be using throughout the remainder of the paper. Here, we introduce $s'$ in order to specify the system of variables in which we are working, as this distinction turns out to be important later.

\begin{proposition} \label{prop:F01}
The following equation holds.
\[
x \frac{\partial}{\partial x} F_{0,1}(x) = P(z)
\]
\end{proposition}

\begin{proof}
With the change of variables of equation~\eqref{eq:variables}, we have $x \frac{\partial}{\partial x} = \frac{z}{1-s'zP'(z)} \frac{\partial}{\partial z}$ and $\frac \partial{\partial s} = \frac \partial{\partial s'} + \frac{zP(z)}{1- s'zP'(z)}\frac \partial{\partial z }$. Following Proposition~\ref{prop:cutandjoin1}, the differential equation satisfied by $F_{0,1}(x)$ is
\[
\frac{\partial}{\partial s} F_{0,1}(x) = \frac{1}{2} \bigg[ x \frac{\partial}{\partial x} F_{0,1}(x) \bigg]^2.
\]
Apply the operator $x\frac \partial{\partial x}$ to both sides and switch to the variables $(z,s')$ to obtain the equation
\[
\bigg( \frac \partial{\partial s'} + \frac{zP(z)}{1- s'zP'(z)}\frac \partial{\partial z} \bigg) x\frac \partial{\partial x} F_{0,1}(x) = \bigg( x\frac \partial {\partial x} F_{0,1}(x) \bigg) \frac z{1- s'zP'(z)}\frac \partial{\partial z} \bigg( x \frac \partial{\partial x} F_{0,1}(x) \bigg).
\]

We have the initial condition $\left. x\frac \partial{\partial x}F_{0,1}(x) \right|_{s' = 0} = P \big( \left. x(z,s') \right|_{s' = 0} \big) = P(z)$. So one can solve the equation perturbatively by introducing an expansion of the form
\[
x\frac \partial{\partial x} F_{0,1}(x) = P(z) + P_1(z) s' + P_2(z) s'^2 + \cdots.
\]
Substituting the perturbative expansion into the equation produces a system of equations, allowing one to explicitly determine the coefficients. Solving the system, we conclude that $P_i(z) = 0$ for every positive integer $i$.
\end{proof}

\begin{corollary}
From the previous proposition, we immediately deduce that $\Omega_{0,1}(x_1) = \frac{y(z_1) \, \dd x(z_1)}{x(z_1)}$. Due to equation~\eqref{eq:trbases}, this verifies the $(g,n) = (0,1)$ case of our main conjecture.
\end{corollary}

Next, we use the cut-and-recursion in the case $(g,n) = (0,2)$.

\begin{proposition} \label{prop:F02}
The following equations hold.
\begin{align*}
\Omega_{0,2}(x_1, x_2) &= \frac{\dd z_1 \, \dd z_2}{(z_1-z_2)^2} - \frac{\dd x_1 \, \dd x_2}{(x_1-x_2)^2} \\
x_1 \frac{\partial}{\partial x_1} F_{0,2}(x_1, x_2) &= -\frac{x_2}{x_1-x_2} + \frac{z_2}{(z_1-z_2) (1 - s' z_1 P'(z_1))}
\end{align*}
\end{proposition}

\begin{proof}
The cut-and-join recursion of Proposition~\ref{prop:cutandjoin1} in the case $(g,n) = (0,2)$ states the following.
\[
\frac \partial{\partial s} F_{0,2}(x_1,x_2) = \sum_{i=1}^2 \bigg[ x_i \frac{\partial}{\partial x_i}F_{0,1}(x_i) \bigg] \bigg[ x_i \frac{\partial}{\partial x_i} F_{0,2}(x_1,x_2) \bigg] + \frac{x_1x_2}{x_1 - x_2}\bigg[ \frac \partial{\partial x_1} F_{0,1}(x_1) - \frac \partial{\partial x_2} F_{0,1}(x_2)\bigg].
\]
From the change of variables of equation~\eqref{eq:variables}
\[
x_i = z_i \exp\left(- s' P(z_i) \right) \qquad \text{and} \qquad s = s',
\]
we have $x_i \frac{\partial}{\partial x_i} = \frac{z_i}{1-s'z_iP'(z_i)} \frac{\partial}{\partial z_i}$ and $\frac \partial{\partial s} = \frac \partial{\partial s'} + \frac{z_1P(z_1)}{1- s'z_1P'(z_1)}\frac \partial{\partial z_1} + \frac{z_2P(z_2)}{1- s'z_2P'(z_2)}\frac \partial{\partial z_2}$. If we furthermore substitute $x_i \frac \partial{\partial x_i} F_{0,1}(x_i) = P(z_i)$ from Proposition~\ref{prop:F01} and simplify, we obtain
\[ 
\frac \partial{\partial s'} F_{0,2}(x_1, x_2) = \frac{x_2 P(z_1) - x_1 P(z_2)}{x_1 - x_2}.
\]
Now integrate and use the fact that $\left.F_{0,2}(x_1, x_2) \right|_{s = s' = 0} = 0$ to obtain
\[
F_{0,2}(x_1, x_2) = -\log\bigg( \frac{x_1 - x_2}{z_1 - z_2} \bigg) - s'P(z_1) - s' P(z_2).
\]
One deduces the first equation of the proposition by applying $\dd_1 \dd_2$ directly to both sides and the second by applying $x_1 \frac{\partial}{\partial x_1}$.
\end{proof}

We may now substitute the calculations of $F_{0,1}$ and $F_{0,2}$ from Propositions~\ref{prop:F01} and~\ref{prop:F02}, respectively, into the cut-and-join recursion. We obtain the following result, which bears a similar structure to that of the topological recursion.

\begin{corollary} \label{cor:cutjoin2}
The free energies satisfy the following equation for $2g-2+n>0$.
\begin{align*}
\bigg[ \frac{\partial}{\partial s} - \sum_{i=1}^n P(z_i) x_i \frac{\partial}{\partial x_i} \bigg] F_{g,n}(x_1, \ldots, x_n) &= \frac{1}{2} \sum_{i=1}^n \bigg[ u_1 u_2 \frac{\partial^2}{\partial u_1 \, \partial u_2} F_{g-1,n+1}(u_1, u_2, \xx_{S \setminus \{i\}}) \bigg]_{u_1=x_i, u_2=x_i} \\
&+ \frac{1}{2} \sum_{i=1}^n \mathop{\sum_{g_1+g_2=g}}_{I \sqcup J = S \setminus \{i\}}^{\mathrm{stable}} \bigg[ x_i \frac{\partial}{\partial x_i} F_{g_1,|I|+1}(x_i, \xx_I) \bigg] \bigg[ x_i \frac{\partial}{\partial x_i} F_{g_2,|J|+1}(x_i, \xx_J) \bigg] \\
&+ \sum_{i\ne j} \frac{z_j}{(z_i-z_j) (1-s'z_iP'(z_i))} x_i \frac{\partial}{\partial x_i} F_{g,n-1}(\xx_{S \setminus \{j\}})
\end{align*}
The word ``stable'' over the inner summation on the second line means that we exclude terms involving $F_{0,1}$ or $F_{0,2}$. Equivalently, using the change of variables $x_i = z_i \exp\left( -s' P(z_i) \right)$ and $s = s'$, we have $x_i \frac \partial{\partial x_i} = \frac{z_i}{1 - s' z_i P'(z_i)} \frac \partial{\partial z_i}$ and $ \frac{\partial}{\partial s} = \frac \partial {\partial s'} + \sum P(z_i) x_i \frac{\partial}{\partial x_i}$. So the free energies satisfy the following equation for $2g-2+n>0$.
\begin{align*}
\frac{\partial}{\partial s'} F_{g,n}(x_1, \ldots, x_n) &= 
 \frac{1}{2} \sum_{i=1}^n \frac{z_i^2}{(1 - s'z_i P'(z_i))^2} \bigg[ \frac{\partial^2}{\partial v_1 \, \partial v_2} F_{g-1,n+1}(u_1, u_2, \xx_{S \setminus \{i\}}) \bigg]_{\substack{u_1=u_2=x_i \\ v_1=v_2=z_1}} \\
&+ \frac{1}{2} \sum_{i=1}^n \mathop{\sum_{g_1+g_2=g}}_{I \sqcup J = S \setminus \{i\}}^{\mathrm{stable}} \frac{z_i^2}{(1 - s'z_i P'(z_i))^2} \bigg[ \frac{\partial}{\partial z_i} F_{g_1,|I|+1}(x_i, \xx_I) \bigg] \bigg[ \frac{\partial}{\partial z_i} F_{g_2,|J|+1}(x_i, \xx_J) \bigg] \\
&+ \sum_{i\ne j} \frac{z_i z_j}{(z_i-z_j) (1-s'z_iP'(z_i))^2} \frac{\partial}{\partial z_i} F_{g,n-1}(\xx_{S \setminus \{j\}})
\end{align*}
\end{corollary}

\subsection{Pruned double Hurwitz numbers}

The notion of pruned Hurwitz numbers was introduced by Norbury and the first author~\cite{do-nor17}. The general idea is to interpret Hurwitz numbers as an enumeration of branching graphs and then to restrict the enumeration to those branching graphs without leaves. The upshot is that pruned Hurwitz numbers store the same information as Hurwitz numbers but are in some sense better behaved. Similar to their unpruned counterparts, pruned Hurwitz numbers also satisfy a cut-and-join recursion and exhibit a polynomial structure.

Okounkov and Pandharipande demonstrated how to associate a branching graph to a branched cover $f: (\Sigma; p_1, \ldots, p_n) \to (\mathbb{CP}^1; \infty)$~\cite{oko-pan}. It is essentially the graph embedded on $\Sigma$ formed by the preimage of the star graph on $\mathbb{CP}^1$. To construct the star graph on $\mathbb{CP}^1$, fix the branch points at the $m$th roots of unity and consider \emph{half-edges} connecting a vertex at 0 to the $m$th roots of unity by line segments. The branching graph is endowed with extra labels --- the faces are labelled with the point $p_i$ that they contain, while the half-edges are labelled by the corresponding root of unity. Note that two half-edges meet to create a \emph{full-edge} precisely at the $m$ simple ramification points on $\Sigma$. These considerations lead naturally to the following definition.

\begin{definition}
A \emph{branching graph} of type $(g,n)$ is a graph comprising half-edges and full-edges meeting at vertices, embedded on a genus $g$ oriented surface $\Sigma$ such that the complement consists of $n$ disks labelled from 1 up to $n$. We furthermore require that
\begin{itemize}
\item the half-edges adjacent to a vertex are cyclically labelled $1, 2, \ldots, m; 1, 2, \ldots, m; \ldots ; 1, 2, \ldots, m$; and
\item there are precisely $m$ full-edges, which are labelled $1, 2, \ldots, m$.
\end{itemize}
Define the \emph{perimeter of a face} and the \emph{degree of a vertex} to be the number of times each label appears on an adjacent half-edge. We consider two branching graphs to be equivalent if there is an orientation-preserving diffeomorphism of their underlying surfaces that maps one branching graph to the other, while preserving all of the labels.
\end{definition}

\begin{proposition} \label{prop:bgraphs}
The double Hurwitz number $DH_{g,n}(\mu_1, \ldots \mu_n)$ is the weighted count of branching graphs of type $(g,n)$ such that the face labelled $i$ has perimeter $\mu_i$. The weight of such a branching graph is
\[
\frac{q_{\lambda_1} q_{\lambda_2} \cdots q_{\lambda_{\ell}}}{|\mathrm{Aut}~ \Gamma|} \frac{s^m}{m!},
\]
where $m$ is the number of full-edges, $\mathrm{Aut}~ \Gamma$ is the automorphism group of the branching graph, and $\lambda_1, \lambda_2, \ldots, \lambda_\ell$ denote the degrees of the vertices.
\end{proposition}

\begin{definition}
A \emph{leaf} of a branching graph is a vertex that is adjacent to exactly one full-edge, where that full-edge is not a loop. Define the \emph{pruned double Hurwitz number} $PH_{g,n}(\mu_1, \ldots \mu_n)$ to be the weighted count of Proposition~\ref{prop:bgraphs}, restricted to those branching graphs without leaves.
\end{definition}

A table of pruned double Hurwitz numbers appears in Appendix~\ref{sec:pdata}.

The pruned viewpoint was applied to double Hurwitz numbers in the recent work of Hahn~\cite{hah}. Hahn derived a cut-and-join recursion for pruned double Hurwitz numbers, using a different notation and normalisation convention to the one we have adopted in this paper. He furthermore demonstrated that they exhibit piecewise polynomial behaviour and derived a pruning correspondence that relates double Hurwitz numbers to their pruned counterparts. We state the pruning correspondence without proof, since it is a restatement of the pruning correspondence in Hahn's paper~\cite[Theorem~3.4]{hah}. 
However, we wish to draw attention to the fact that it is expressed very naturally using our packaging of the double Hurwitz numbers.

\begin{proposition}[Pruning correspondence, version 1]
For $(g,n) \neq (0,1)$, we have the following equations.
\begin{align*}
DH_{g,n}(\mu_1, \ldots, \mu_n) &= \sum_{\nu_1, \ldots, \nu_n=1}^{\mu_1, \ldots, \mu_n} PH_{g,n}(\nu_1, \ldots, \nu_n) \prod_{i=1}^n C(\mu_i, \nu_i), & \text{where } C(\mu, \nu) = \frac{\nu}{\mu} [z^{\mu-\nu}] \exp(\mu sP(z)) \\
PH_{g,n}(\nu_1, \ldots, \nu_n) &= \sum_{\mu_1, \ldots, \mu_n=1}^{\nu_1, \ldots, \nu_n} DH_{g,n}(\mu_1, \ldots, \mu_n) \prod_{i=1}^n \widehat{C}(\nu_i, \mu_i), & \text{where } \widehat{C}(\nu, \mu) = \frac{\mu}{\nu} [z^{\nu-\mu}] \frac{1-szP'(z)}{\exp(\mu sP(z))}
\end{align*}
\end{proposition}

Note that $C(\mu, \nu)$ is a weighted homogeneous polynomial in $sq_1, sq_2, \ldots$ with positive rational coefficients, while $\widehat{C}(\nu, \mu)$ is a weighted homogeneous polynomial in $sq_1, sq_2, \ldots$ with rational coefficients.

In fact, the pruning correspondence can be even more succinctly stated in the language of the free energies. In particular, the pruned double Hurwitz numbers arise as coefficients in the expansion with respect to the rational parameters $z_1, \ldots, z_n$.

\begin{corollary}[Pruning correspondence, version 2]
For $(g,n) \neq (0,1)$, we have the following equations.
\begin{align*}
F_{g,n} &= \sum_{\nu_1, \ldots, \nu_n=1}^\infty PH_{g,n}(\nu_1, \ldots, \nu_n) \prod_{i=1}^n z_i^{\nu_i} \\
\Omega_{g,n} &= \sum_{\nu_1, \ldots, \nu_n=1}^\infty PH_{g,n}(\nu_1, \ldots, \nu_n) \prod_{i=1}^n \nu_i z_i^{\nu_i-1} \dd z_i
\end{align*}
\end{corollary}

\begin{proof}
We simply compute the coefficients of $\Omega_{g,n}$ in the $z_i$-expansion.
\begin{align*}
\mathop{\text{Res}}_{z_1=0} \cdots \mathop{\text{Res}}_{z_n=0} \Omega_{g,n} \prod_{i=1}^n z_i^{-\nu_i} &= \mathop{\text{Res}}_{z_1=0} \cdots \mathop{\text{Res}}_{z_n=0} \sum_{\mu_1, \ldots, \mu_n=1}^\infty DH_{g,n}(\mu_1, \ldots, \mu_n) \prod_{i=1}^n \mu_i z_i^{-\nu_i} x_i^{\mu_i-1} \dd x_i \\
&= \sum_{\mu_1, \ldots, \mu_n=1}^\infty DH_{g,n}(\mu_1, \ldots, \mu_n) \prod_{i=1}^n \mu_i [z_i^{\nu_i-1}] x_i^{\mu_i-1} \frac{\dd x_i}{\dd z_i} \\
&= \sum_{\mu_1, \ldots, \mu_n=1}^\infty DH_{g,n}(\mu_1, \ldots, \mu_n) \prod_{i=1}^n \mu_i [z_i^{\nu_i-\mu_i}] \frac{1-sz_iP'(z_i)}{\exp(\mu_isP(z_i))} \\
&= \sum_{\mu_1, \ldots, \mu_n=1}^{\nu_1, \ldots, \nu_n} DH_{g,n}(\mu_1, \ldots, \mu_n) \prod_{i=1}^n \nu_i \widehat{C}(\nu_i, \mu_i) = PH_{g,n}(\nu_1, \ldots, \nu_n) \prod_{i=1}^n \nu_i
\end{align*}
The desired statement for the multidifferentials $\Omega_{g,n}$ follows immediately, while the desired statement for the free energies $F_{g,n}$ can be obtained by integrating.
\end{proof}

We do not state here the cut-and-join recursion for pruned double Hurwitz numbers. However, we remark that it can be obtained either combinatorially, as performed by Hahn~\cite{hah}, or equivalently by expressing Corollary~\ref{cor:cutjoin2} at the level of coefficients in the $z_i$-expansion.

In Section~\ref{sec:proof}, we will reduce our main conjecture to a weaker conjecture concerning the structure of double Hurwitz numbers. We have introduced pruned double Hurwitz numbers since enumeration of pruned objects in general can help to reveal structure of the underlying enumeration~\cite{do-nor17}.

\section{Evidence for the conjecture} \label{sec:evidence}

\subsection{The quantum curve} \label{subsec:qcurve}

It is known that there is a quantum curve underlying many instances of the topological recursion and conversely, the existence of a quantum curve can predict spectral curves and topological recursion~\cite{nor16}. We now consider the quantum curve for double Hurwitz numbers and propose it as evidence towards the main conjecture.

Informally, a quantum curve for a spectral curve $P(x, y) = 0$ is a certain operator that can be expressed in terms of a non-commutative polynomial $\widehat{P}(\widehat{x}, \widehat{y})$ in the multiplication operator $\widehat{x} = x$ and the differential operator $\widehat{y} = \h x \frac{\partial}{\partial x}$ such that\footnote{The choice of polarisation --- in other words, the operators $\widehat{x}$ and $\widehat{y}$ --- relates to the fact that we are dealing with a $\mathbb{C}^* \times \mathbb{C}$ spectral curve. For a $\mathbb{C} \times \mathbb{C}$ spectral curve, it is common to use $\widehat{y} = \h \frac{\partial}{\partial x}$. In both cases though, we have the canonical commutation relation $[ \widehat{x}, \widehat{y}] = -\h$.}
\begin{itemize}
\item the semi-classical limit of $\widehat{P}(\widehat{x}, \widehat{y})$ recovers $P(x, y)$; and
\item $\widehat{P}(\widehat{x}, \widehat{y})$ annihilates the wave function.
\end{itemize}
The semi-classical limit is obtained by sending $\widehat{x}$ and $\widehat{y}$ to the commuting variables $x$ and $y$, while also sending $\h$ to 0. The wave function is a certain generating function constructed from the coefficients of the correlation differentials obtained from topological recursion. We define the wave function only in the case of double Hurwitz numbers and point the reader to the literature for more information on quantum curves~\cite{nor16}.

In the case of double Hurwitz numbers, the \emph{partition function} and \emph{wave function} are generating functions defined by the following formulas, respectively.
\begin{align*}
Z(p_1, p_2, \ldots; \h) &= \exp \bigg[ \sum_{g=0}^\infty \sum_{n=1}^\infty \sum_{\mu_1, \ldots, \mu_n=1}^\infty DH_{g,n}(\mu_1, \ldots, \mu_n) \frac{\h^{2g-2+n}}{n!} p_{\mu_1} \cdots p_{\mu_n} \bigg] \\
\psi(x, \h) &= \left. Z(p_1, p_2, \ldots; \h) \right|_{p_i=x^i}
\end{align*}

\begin{proposition}[Alexandrov--Lewanski--Shadrin \cite{ale-lew-sha}]
The wave function $\psi(x, \h)$ for the double Hurwitz numbers satisfies the quantum curve equation $\mathcal{Q} \, \psi(x,\hbar) = 0$, where
\[
\mathcal{Q} = \widehat{y} - \sum_{k=1}^d q_k \exp \big( \tfrac{1}{2} s\hbar k (k-1) \big) \widehat{x}^k \exp (sk \widehat{y}).
\]
\end{proposition}

One can immediately observe that the semi-classical limit of the operator $\mathcal Q$ is simply $y - P(x \exp(sy))$, for which we have the rational parametrisation of equation~\eqref{eq:scurve}. The existence of the quantum curve lends strong support to Conjecture~\ref{con:main}.

\begin{remark}
As previously mentioned, the quantum curve may be used to predict the spectral curve for a topological recursion. It is also often the case that the $(0,1)$ information of a problem may be used to predict the spectral curve for a topological recursion~\cite{dum-mul-saf-sor}. We remark here that the quantum curve may be used to pass directly to the $(0,1)$ information. In the following, we use the quantum curve to provide another proof of Proposition~\ref{prop:F01}, which states that $x\frac \partial {\partial x} F_{0,1} (x)= P(z)$.

We begin with the two observations
\[
\psi(x,\h) = \exp \bigg( \frac 1\hbar F_{0,1}(x) + O(1) \bigg) \qquad \text{and} \qquad \widehat{y} \psi(x,\h) = \bigg( x\frac \partial {\partial x} F_{0,1}(x) + O(\h)\bigg)\psi(x,\h),
\]
which follow directly from the definition of the wave function. Therefore, we have the equation
\[
\exp \bigg( \frac 12 s\h k(k-1) \bigg) \widehat{x}^k \exp ( sk\widehat{y} ) \psi(x,\h) = ( 1 + O(\h) )x^k \bigg(e^{skx\frac \partial {\partial x} F_{0,1}(x)} + O(\h)\bigg)\psi(x,\h).
\]
The vanishing of $\mathcal Q \, \psi(x,\h)$ implies the following equation on $x\frac \partial {\partial x} F_{0,1}(x)$.
\[
x\frac \partial {\partial x} F_{0,1}(x) - \sum_{k=1}^d q_k x^k e^{s k x\frac \partial {\partial x} F_{0,1}(x)} = 0
\]

Now use the substitution $x = z \exp \big(-sx\tfrac \partial{\partial x}F_{0,1}(x)\big)$, where $z$ is a new variable and $x\frac \partial{\partial x} F_{0,1}(x)$ in the exponent is considered as a function of $z$. Then the previous equation recovers Proposition~\ref{prop:F01}, namely
\[
x \frac \partial{\partial x} F_{0,1}(x) = \sum_{k = 1}^d q_k z^k.
\]
\end{remark}

\subsection{Low genus evidence} \label{subsec:lowgenus}

One can attempt to verify the main conjecture via direct computation in low genus. We include below the calculations in the case $(g,n) = (0,3)$ and $(1,1)$.

\subsubsection*{Calculation of $F_{0,3}$}

The cut-and-join recursion of Corollary~\ref{cor:cutjoin2} in the case $(g,n) = (0,3)$ reads
\[
\frac \partial{\partial s'}F_{0,3} = \sum_{\text{cyclic}} \frac{x_1x_2}{x_1 - x_2} \bigg[ \frac{\partial}{\partial x_1} F_{0,2}(x_1,x_3) - \frac \partial{\partial x_2} F_{0,2}(x_2,x_3) \bigg] + \sum_{\text{cyclic}} \bigg[ x_1\frac{\partial}{\partial x_1}F_{0,2}(x_1,x_2)\bigg] \bigg[ x_1 \frac{\partial}{\partial x_1} F_{0,2}(x_1,x_3)\bigg].
\]
Substituting the formula of Proposition~\ref{prop:F02} yields great simplification and we are ultimately left with
\[
\frac \partial{\partial s'}F_{0,3} = -1 + \sum_{\text{cyclic}} \frac{z_2z_3}{(z_1 - z_2)(z_1 - z_3)}\frac 1{(1- s'z_1P'(z_1))^2}.
\]
One can integrate this equation, using the initial conditions to deduce the constant of integration, in order to obtain
\begin{equation} \label{eq:F03cyclic}
F_{0,3}(x_1, x_2, x_3) = -s' + \sum_{\text{cyclic}} \frac{s'z_2z_3}{(z_1 - z_2)(z_1 - z_3)(1-s'z_1P'(z_1))}.
\end{equation}

From the previous equation, $F_{0,3}$ can be expressed as follows.
\[
F_{0,3}(x_1, x_2, x_3) = 
\frac {s' \det \begin{bmatrix} 
s'z_1P'(z_1) & s'z_2P'(z_2) & s'z_3P'(z_3) \\ 
z_1(1 - s'z_1P'(z_1)) & z_2(1 - s'z_2 P'(z_2)) & z_3(1 - s'z_3 P'(z_3)) \\
 z_1^2(1 - s'z_1P'(z_1)) & z_2^2(1 - s'z_2P'(z_2)) & z_3^2(1 - s'z_3P'(z_3))
\end{bmatrix} }{(1 - s'z_1P'(z_1))(1 - s'z_2P'(z_2))(1 - s'z_3P'(z_3))(z_1 - z_2)(z_2 - z_3)(z_3 - z_1)}
\]
Observe that the determinant in the numerator is divisible by $(z_1-z_2) (z_2-z_3) (z_3-z_1)$, so that $F_{0,3}(x_1, x_2, x_3)$ does not have poles along the diagonals $z_i - z_j$. Now consider the expression
\[
F_{0,3}(x_1, x_2, x_3) \prod_{i=1}^3 (1-s'z_iP'(z_i)) = \frac{\text{Pol}(z_1, z_2, z_3, s)}{(z_1-z_2) (z_2-z_3) (z_3-z_1)}.
\]
We see that the degree of $z_1$ in $\text{Pol}(z_1, z_2, z_3, s)$ is $d+2$. However, we have already deduced that $\text{Pol}(z_1, z_2, z_3)$ is divisible by $(z_1-z_2) (z_2-z_3) (z_3-z_1)$. It follows that the entire expression is in fact a polynomial of degree $d$ in $z_1$. Similarly, one can deduce that the entire expression has degree $d$ in $z_2$ and $z_3$. Now as the determinant in the numerator is divisible by $z_1z_2z_3$, there must exist $C_{i_1, i_2, i_3} \in \mathbb{Q}[s, q_1, q_2, \ldots]$ such that
\begin{equation}
F_{0,3}(x_1, x_2, x_3) = \prod_{i=1}^3 \frac{1}{1-s'z_iP'(z_i)} \times \sum_{i_1, i_2, i_3 =1}^d C_{i_1, i_2, i_3} z_1^{i_1} z_2^{i_2} z_3^{i_3}. \label{eq:F03}
\end{equation}

Note in particular that $F_{0,3}(x_1, x_2, x_3)$ is a rational function in $z_1, z_2, z_3$ with poles only at the branch points of the spectral curve of equation~\eqref{eq:scurve}. We may now check that the topological recursion recovers the correct correlation differential $\omega_{0,3}$ predicted by our main conjecture. We start with the following formula for $\omega_{0,3}$~\cite{eyn-ora07a}.
\begin{align}
\omega_{0,3}(z_1, z_2, z_3) &= - \sum_{i=1}^d \mathop{\text{Res}}_{z=a_i} \frac{\omega_{0,2}(z, z_1) \, \omega_{0,2}(z, z_2) \, \omega_{0,2}(z, z_3) \, x(z)}{\dd x(z) \, \dd y(z)} \nonumber \\
&= \dd z_1 \, \dd z_2 \, \dd z_3 \sum_{i=1}^d \frac{-sa_i^3}{(z_1-a_i)^2 \, (z_2-a_i)^2 \, (z_3-a_i)^2 \, (1 + s a_i^2 P''(a_i))} \label{eq:w03}
\end{align}
Observe that this equation decomposes $\omega_{0,3}(z_1, z_2, z_3)$ into a sum of its principal parts with respect to the variable $z_1$~\cite[Proposition 16]{do-lei-nor}. Recall that the principal part of a meromorphic form may be defined via the formula
\begin{equation} \label{eq:ppart}
\left[ \omega(z_1) \right]_a = \mathop{\text{Res}}_{z=a} \frac{\dd z_1}{z_1-z} \omega(z).
\end{equation}

Now we use this formula to evaluate the principal part of $\Omega_{0,3}(x_1, x_2, x_3) = \dd_1 \dd_2 \dd_3 F_{0,3}(x_1, x_2, x_3)$ at the branch point $a_i$.
\begin{align*}
\frac{\left[ \Omega_{0,3}(x_1, x_2, x_3) \right]_{a_i}}{\dd z_1 \, \dd z_2 \, \dd z_3} &= \frac{1}{\dd z_1 \, \dd z_2 \, \dd z_3} \mathop{\text{Res}}_{z=a_i} \frac{\dd z_1}{z_1-z} \dd \dd_2 \dd_3 F_{0,3}(x, x_2, x_3) \\
&= \frac{\partial}{\partial z_1} \frac{\partial}{\partial z_2} \frac{\partial}{\partial z_3} \mathop{\text{Res}}_{z=a_i} \frac{\dd z}{z_1-z} F_{0,3}(x, x_2, x_3) \\
&= \frac{\partial}{\partial z_1} \frac{\partial}{\partial z_2} \frac{\partial}{\partial z_3} \mathop{\text{Res}}_{z=a_i} \frac{\dd z}{z_1-z} \bigg[ \frac{s'z_2z_3}{(z - z_2)(z - z_3)(1-s'zP'(z))} \bigg]
\end{align*}

The first equality uses the definition of the principal part, the second exchanges the order of taking the residue and the derivative, while the third uses equation~\eqref{eq:F03cyclic}. Note that we have only kept the summand in equation~\eqref{eq:F03cyclic} that will contribute to the residue. We may now evaluate the residue and derivatives explicitly, noting that the expression has a simple pole at $a_i$.

\begin{align}
\frac{\left[ \Omega_{0,3}(x_1, x_2, x_3) \right]_{a_i}}{\dd z_1 \, \dd z_2 \, \dd z_3} &= \frac{\partial}{\partial z_1} \frac{\partial}{\partial z_2} \frac{\partial}{\partial z_3} \mathop{\text{lim}}_{z\to a_i} \frac{1}{z_1-z} \bigg[ \frac{s'z_2z_3 (z-a_i)}{(z - z_2)(z - z_3)(1-s'zP'(z))} \bigg] \nonumber \\
&= \frac{\partial}{\partial z_1} \frac{\partial}{\partial z_2} \frac{\partial}{\partial z_3} \frac{z_2z_3}{(z_1-a_i)(z_2-a_i)(z_3-a_i) (a_iP''(a_i) + P'(a_i))} \nonumber \\ 
&= - \frac{sa_i^3}{(z_1-a_i)^2 (z_2-a_i)^2 (z_3-a_i)^2 (1+sa_i^2P''(a_i))} \label{eq:F03final}
\end{align}

The first equality evaluates the residue via limit, the second is an application of L'H\^{o}pital's rule, while the third evaluates the derivatives explicitly and uses the fact that $1 - s a_i P'(a_i) = 0$.

Comparing equations~\eqref{eq:w03} and~\eqref{eq:F03final}, we see that the principal parts of $\omega_{0,3}$ and $\Omega_{0,3}$ match precisely at the branch points. Furthermore, we know that they are both rational multidifferentials with poles only at the branch points, so we may conclude that they are equal. It follows that Conjecture~\ref{con:main} holds in the case $(g,n) = (0,3)$.

\subsubsection*{Calculation of $F_{1,1}$}
The cut-and-join recursion of Corollary~\ref{cor:cutjoin2} in the case $(g,n) = (1,1)$ reads
\[
\bigg[ \frac \partial{\partial s} - P(z) x \frac{\partial}{\partial x} \bigg] F_{1,1}(x) = \frac 12 \bigg[ u_1 u_2 \frac{\partial^2}{\partial u_1 \partial u_2} F_{0,2}(u_1,u_2) \bigg]_{\substack{u_1=x \\ u_2=x}}.
\]
Passing to $(z,s')$ coordinates and substituting the explicit expression for $\frac{\partial^2}{\partial z_1 \partial z_2}F_{0,2}$ from Proposition~\ref{prop:F02}
yields
\[
\frac \partial {\partial s'} F_{1,1}(x) = \frac{1}{2} \bigg[ \frac{z_1}{1 - s' z_1 P'(z_1)}\frac{z_2}{1- s' z_2 P'(z_2)} \frac{(x(z_1) - x(z_2))^2 - (z_1 - z_2)^2 x'(z_1) x'(z_2)}{(z_1 - z_2)^2 \, (x(z_1) - x(z_2))^2}\bigg]_{z_1 = z_2 = z}.
\]

One can evaluate the right side of this equation using L'H\^opital's rule to obtain
\[
\frac \partial {\partial s'} F_{1,1}(x) = \frac{ z^2\left(3x''(z)^2 - 2x'(z)x'''(z)\right)}{24 \left(1-s'zP'(z)\right)^2x'(z)^2}.
\]

Now integrate this equation, using the initial conditions to deduce the constant of integration.
\[
F_{1,1}(x) = \frac{s'^2 z^2 \left( 3P'' + zP''' + 3s' P'^2 - s'P'^2 - s'zP'P'' + s'z^2P''^2 - s'z^2P'P''' - s'^2 z P'^3 \right)}{24 (1 - s'zP')^3}.
\]

One can explicitly verify that
\begin{equation}\label{eq:F11}
F_{1,1}(x) = \frac{s'^2}{24} \bigg( \frac z{1 - s'zP'(z)}\frac \partial{\partial z }\bigg)^2[zP'(z)] - \frac{s'^2}{24} \bigg( \frac z{1 - s'zP'(z)}\frac \partial{\partial z }\bigg) [P(z)].
\end{equation}

Note that $F_{1,1}(x)$ is a rational function in $z$ with poles only at the branch points of the spectral curve of equation~\eqref{eq:scurve}. Rather than verify that $\Omega_{1,1}(x) = \dd F_{1,1}(x)$ coincides with the correlation differential $\omega_{1,1}(z)$, we propose a more general approach in the next section. We will prove that our main conjecture follows if the free energies satisfy so-called linear loop equations.

\section{Towards a proof of the conjecture} \label{sec:proof}

\subsection{The structure of double Hurwitz numbers}

There is a certain technique for starting with a combinatorial recursion of cut-and-join type and deducing the topological recursion that works in several known cases. For instance, it was successfully employed in the case of single Hurwitz numbers~\cite{eyn-mul-saf}, orbifold Hurwitz numbers~\cite{do-lei-nor, bou-her-liu-mul} and monotone Hurwitz numbers~\cite{do-dye-mat}. One further ingredient usually enters into these proofs and that is a polynomial structure theorem for the enumerative problem. However, such a result is not known in the case of double Hurwitz numbers. In this section, we present a conjecture on the structure of double Hurwitz numbers and then proceed to show that it can be used to deduce our main conjecture.

It is known that the correlation differentials produced by the topological recursion satisfy so-called \emph{linear loop equations}~\cite{bor-eyn-ora}. These assert that at a branch point $a$ of the spectral curve with associated involution $\sigma$, the correlation differentials $\omega_{g,n}$ for $2g-2+n>0$ satisfy the following condition: the sum
\[
\omega_{g,n}(z_1, z_2, \ldots, z_n ) + \omega_{g,n}(\sigma(z_1), z_2, \ldots, z_n)
\]
is analytic at $z_1=a$. The linear loop equations imply structure to the underlying coefficients of the correlation differentials. For example, in the case of single Hurwitz numbers, they are equivalent to the fact that
\[
H_{g,n}(\mu_1, \ldots, \mu_n) = \prod_{i=1}^n \frac{\mu_i^{\mu_i}}{\mu_i!} P_{g,n}(\mu_1, \ldots, \mu_n),
\]
where $P_{g,n}$ is a symmetric polynomial, a fact that follows from the ELSV formula. The linear loop equations are equivalent to an analogous quasi-polynomial structure in the case of orbifold Hurwitz numbers. As previously mentioned, an analogous result is not known in the case of double Hurwitz numbers. However, given our main conjecture, it is natural to posit that the following is true.

\begin{conjecture} \label{con:poly}
For $2g-2+n>0$, 
\begin{itemize}
\item $F_{g,n}(z_1, \ldots, z_n)$ is a rational function with poles only at $z_i = a_j$ for $i = 1, 2, \ldots, n$ and $j = 1, 2, \ldots, d$; and
\item $F_{g,n}(z_1, z_2, \ldots, z_n) + F_{g,n}(\sigma_i(z_1), z_2, \ldots, z_n)$ is analytic at $z_1=a_i$ for $i = 1, 2, \ldots, d$. 
\end{itemize}
We will subsequently refer to these constraints as \emph{linear loop equations}.
\end{conjecture}

From the general theory of topological recursion, we know that Conjecture~\ref{con:main} implies Conjecture~\ref{con:poly}. In fact, we will prove in Theorem~\ref{thm:main} that the converse is also true. Therefore, the linear loop equations are equivalent to the statement that topological recursion governs the double Hurwitz numbers.

\begin{definition}
Define the $\mathbb{C}(s)$-vector space $V(z)$, whose elements are rational functions $p(z)$ such that
\begin{itemize}
\item $p(z)$ has poles only at the branch points $a_1, a_2, \ldots, a_d$; and
\item $p(z) + p(\sigma_i(z))$ is analytic at $z = a_i$ for $i = 1, 2, \ldots, d$.
\end{itemize}
\end{definition}

\begin{lemma} \label{lem:basis}
A basis for $V(z)$ is formed by $\phi_k^{i}(z)$ for $i = 1, 2, \ldots, d$ and $k = 0, 1, 2, \ldots$, where we define $\phi_{-1}^{i}(z) = z^i$ and
\[
\phi_{k+1}^{i} = \frac{z}{1-szP'(z)} \frac{\partial}{\partial z} \phi_k^{i}(z).
\]
\end{lemma}

One can find a short proof of the lemma in the literature~\cite[Lemma 14]{do-lei-nor}. Given the fact that $F_{g,n}(z_1, \ldots, z_n)$ is symmetric in its arguments, it satisfies the linear loop equations if and only if
\[
F_{g,n}(z_1, \ldots, z_n) \in V(z_1) \otimes V(z_2) \otimes \cdots \otimes V(z_n).
\]

\begin{proposition}
The linear loop equations hold for $(g,n) = (0,3)$ and $(1,1)$.
\end{proposition}

\begin{proof}
From equation~\eqref{eq:F03}, we have
\[
F_{0,3}(x_1, x_2, x_3) = \sum_{i_1, i_2, i_3 =1}^d \frac{C_{i_1,i_2,i_3}}{i_1 i_2 i_3} \phi_0^{i_1}(z_1) \, \phi_0^{i_2}(z_2) \, \phi_0^{i_3}(z_3) \in V(z_1) \otimes V(z_2) \otimes V(z_3).
\]
From equation~\eqref{eq:F11}, we have
\[
F_{1,1}(x) = \frac{s^2}{24} \sum_{i=1}^d \left[ iq_i \phi_1^i(z) - q_i \phi_0^i(z) \right] \in V(z). \qedhere
\]
\end{proof}

\begin{remark} \label{rem:commuting}
By Lemma~\ref{lem:basis}, we know that $x \frac{\partial}{\partial x} = \frac{z}{1-s'zP'(z)} \frac{\partial}{\partial z}$ preserves $V(z)$ in the sense that for any $p(z) \in V(z)$, we also have $x \frac{\partial}{\partial x} p(z) \in V(z)$. Since $x$ and $s$ are commuting variables, we also know that $\frac{\partial}{\partial s} = \frac{\partial}{\partial s'} + \frac{zP(z)}{(1-s'zP'(z)} \frac{\partial}{\partial z}$ preserves $V(z)$. This seemingly trivial remark will prove to be useful later.
\end{remark}

The linear loop equations imply structure on the free energies and hence, also on the double Hurwitz numbers themselves. We have $F_{g,n}(z_1, \ldots, z_n) \in V(z_1) \otimes V(z_2) \otimes \cdots \otimes V(z_n)$, supposing that the linear loop equations hold. So there exist $C_{g,n} \big( \substack{i_1, \ldots, i_n \\ m_1, \ldots, m_n} \big)$ that depend on $s, q_1, q_2, \ldots$ for $1 \leq i_1, \ldots, i_n \leq d$ and $m_1, \ldots, m_n \geq 0$ such that the following is true.
\begin{align*}
F_{g,n}(z_1, \ldots, z_n) &= \sum_{i_1, \ldots, i_n = 1}^d \sum_{m_1, \ldots, m_n=0}^{\mathrm{finite}} C_{g,n} \big( \substack{i_1, \ldots, i_n \\ m_1, \ldots, m_n} \big) \prod_{k=1}^n \phi_{m_k}^{i_k}(z_k) \\
&= \sum_{i_1, \ldots, i_n = 1}^d \sum_{m_1, \ldots, m_n=0}^{\mathrm{finite}} C_{g,n} \big( \substack{i_1, \ldots, i_n \\ m_1, \ldots, m_n} \big) \prod_{k=1}^n \bigg( x_k \frac{\partial}{\partial x_k} \bigg)^{m_k+1} z_k^{i_k}
\end{align*}

Write $z^i = \sum_{\mu=0}^\infty A_\mu^i x^\mu$ and observe that one can deduce the coefficients $A_\mu^i$ via the Lagrange inversion formula. This then leads to the following conjectural expression for the double Hurwitz numbers, which is in fact equivalent to  Conjecture~\ref{con:poly}.

\begin{conjecture} \label{con:elsv}
For $2g-2+n>0$, there exist $C_{g,n} \big( \substack{i_1, \ldots, i_n \\ m_1, \ldots, m_n} \big)$ independent of $\mu_1, \ldots, \mu_n$ such that the double Hurwitz numbers may be expressed as
\[
DH_{g,n}(\mu_1, \ldots, \mu_n) = \sum_{i_1, \ldots, i_n = 1}^d \prod_{k=1}^n A_{\mu_k}^{i_k} \sum_{m_1, \ldots, m_n=0}^{\mathrm{finite}} C_{g,n} \big( \substack{i_1, \ldots, i_n \\ m_1, \ldots, m_n} \big) \prod_{k=1}^n \mu_k^{m_k+1},
\]
where
\[
A_{\mu}^i = i \sum_{|\lambda| = \mu-i} \frac{\mu^{\ell(\lambda)-1}}{|\mathrm{Aut}~\lambda|} q_{\lambda_1} q_{\lambda_2} \cdots q_{\lambda_{\ell(\lambda)}} s^{\ell(\lambda)}.
\]
Here, $\lambda$ represents an integer partition with $\ell(\lambda)$ parts and $\mathrm{Aut}~\lambda$ is the set of permutations of the tuple $(\lambda_1, \lambda_2, \ldots, \lambda_{\ell(\lambda)})$ that leave it invariant.
\end{conjecture}

\begin{remark}
The previous conjecture reduces to the polynomial structure of single Hurwitz numbers in the case $P(z) = z$ and the quasi-polynomial structure of orbifold Hurwitz numbers in the case $P(z) = z^a$. We remark that the proof of polynomiality for single Hurwitz numbers was originally achieved via the ELSV formula, a rather deep result from algebraic geometry~\cite{eke-lan-sha-vai}. A second proof arose via a thorough analysis of the infinite wedge space expression for single Hurwitz numbers~\cite{dun-kaz-ora-sha-spi}. In the context of double Hurwitz numbers, the previously mentioned conjecture of Goulden, Jackson and Vakil may play the role of the ELSV formula, although the underlying algebraic geometry is unclear at present. Thus, one may hope that the infinite wedge space technology may lead to progress the combinatorial structure of double Hurwitz numbers and thence towards the geometry of double Hurwitz numbers.
\end{remark}

\subsection{From cut-and-join recursion to topological recursion}

At the start of the section, we alluded to a certain technique for starting with a combinatorial recursion of cut-and-join type and deducing the topological recursion. We now carry this out in the context of double Hurwitz numbers, using the linear loop equations as an unproven assumption. In other words, the rest of the paper is dedicated to proving the following result.

\begin{theorem} \label{thm:main}
Conjecture~\ref{con:poly} implies Conjecture~\ref{con:main}.
\end{theorem}

\begin{proof}[Proof of Theorem~\ref{thm:main}]
Consider the cut-and-join recursion of Corollary~\ref{cor:cutjoin2} for $2g-2+n > 1$. We will use the inductive hypothesis that $\Omega_{g',n'} = \omega_{g',n'}$ for $2g'-2+n' < 2g-2+n$. We will adopt a gentle abuse of notation and write $F_{g,n}(z_1, \ldots, z_n)$ for the rational function obtained by taking $F_{g,n}(x_1, \ldots, x_n)$ and changing coordinates to the $z$ variables. Recall that it is rational due to the assumption of the linear loop equations of Conjecture~\ref{con:poly}.
\begin{align*}
\frac{\partial}{\partial s'} F_{g,n}(z_1, \ldots, z_n) &= 
 \frac{1}{2} \sum_{i=1}^n \frac{z_i^2}{(1 - s'z_i P'(z_i))^2} \bigg[ \frac{\partial^2}{\partial u_1 \, \partial u_2} F_{g-1,n+1}(u_1, u_2, \zz_{S \setminus \{i\}}) \bigg]_{u_1=z_i, u_2=z_i} \\
 \\
&+ \frac{1}{2} \sum_{i=1}^n \mathop{\sum_{g_1+g_2=g}}_{I \sqcup J = S \setminus \{i\}}^{\mathrm{stable}} \frac{z_i^2}{(1 - s'z_i P'(z_i))^2}\bigg[ \frac{\partial}{\partial z_i} F_{g_1,|I|+1}(z_i, \zz_I) \bigg] \bigg[ \frac{\partial}{\partial z_i} F_{g_2,|J|+1}(z_i, \zz_J) \bigg] \\
&+ \sum_{i\ne j} \frac{z_i z_j}{(z_i-z_j) (1-s'z_iP'(z_i))^2} \frac{\partial}{\partial z_i} F_{g,n-1}(\zz_{S \setminus \{j\}})
\end{align*}

Pick a branch point $a_i$ of the spectral curve with associated local involution $\sigma_i$. The general strategy is to take the principal part of both sides at $z_1 = a_i$. As we have done previously, denote the principal part by the notation $[ ~\cdot~ ]_{a_i}$. We take the result, apply $\sigma_i$ to $z_1$ to both sides of the equation, then add it to the original. We denote this procedure by the notation $\mathrm{Sym}_i[ f(z_1) ] = f(z_1) + f(\sigma_i(z_1))$. Note that applying $\mathrm{Sym}_i[ ~\cdot~ ]_{a_i}$ kills all terms that satisfy the linear loop equations with respect to $z_1$. Thus, we obtain the following equation.
\begin{align*}
&\mathrm{Sym}_i  \bigg[ \frac{\partial}{\partial s'} F_{g,n}(z_1, \ldots, z_n) \bigg]_{a_i} = 
 \frac{1}{2} \, \mathrm{Sym}_i\left[ \frac{z_1^2}{(1 - s'z_1 P'(z_1))^2}\bigg[ \frac{\partial^2}{\partial u_1 \, \partial u_2} F_{g-1,n+1}(u_1, u_2, \zz_{S \setminus \{1\}}) \bigg]_{u_1=z_1, u_2=z_1} \right]_{a_i} \\
&\qquad + \frac{1}{2} \mathop{\sum_{g_1+g_2=g}}_{I \sqcup J = S \setminus \{1\}}^{\mathrm{stable}} \mathrm{Sym}_i\left[ \frac{z_1^2}{(1 - s'z_1 P'(z_1))^2}\bigg[ \frac{\partial}{\partial z_1} F_{g_1,|I|+1}(z_1, \zz_I) \bigg] \bigg[ \frac{\partial}{\partial z_1} F_{g_2,|J|+1}(z_1, \zz_J) \bigg] \right]_{a_i} \\
&\qquad + \sum_{j=2}^n \mathrm{Sym}_i\left[ \frac{z_1z_j}{z_1 - z_j} \bigg[ \frac 1{(1-z_1s'P'(z_1))^2} \frac{\partial}{\partial z_1} F_{g,n-1}(\zz_{S \setminus \{j\}}) - \frac 1{(1-z_js'P'(z_j))^2} \frac{\partial}{\partial z_j} F_{g,n-1}(\zz_{S \setminus \{1\}})\bigg] \right]_{a_i}
\end{align*}

Now if $f_i(z_1)$ and $g(z_1)$ satisfy the linear loop equation, then $f(z_1) + f(\sigma(z_1))$ and $g(z_1) + g(\sigma(z_1))$ are analytic at $a_i$, so their product is as well. So we have the following equations.
\[
 \sum_k [f_k(z_1) g_k(z_1) + f_k(\sigma(z_1)) g_k(z_1) + f_k(z_1) g_k(\sigma(z_1)) + f_k(\sigma(z_1)) g_k(\sigma(z_1))]_{a_i} = 0
\]
\[
\sum_k [f_k(z_1) g_k(z_1) + f_k(\sigma(z_1)) g_k(\sigma(z_1))]_{a_i} = - \sum_k [f_k(\sigma(z_1)) g_k(z_1) + f_k(z_1) g_k(\sigma(z_1))]_{a_i}
\]
\[
\mathrm{Sym}_i \Big[ \sum_k f_k(z_1) g_k(z_1)) \Big]_{a_i} = - \mathrm{Sym}_i \Big[ \sum_k f_k(\sigma(z_1)) g_k(z_1) \Big]_{a_i}
\]

Use this on the first and second lines of the right side, keeping in mind that if $f(z)$ satisfies the linear loop equation, then so does $x \frac{\partial}{\partial x} f(z)$. 
We also rewrite the third line in a more convenient form.
\begin{align*}
&\mathrm{Sym}_i\bigg[ \frac{\partial}{\partial s'} F_{g,n}(z_1, \ldots, z_n) \bigg]_{a_i} = -
 \frac{1}{2} \, \mathrm{Sym}_i\left[ x_1^2 \bigg[ \frac{\partial^2}{\partial u_1 \, \partial u_2} F_{g-1,n+1}(z(u_1), \sigma(z(u_2)), \zz_{S \setminus \{1\}}) \bigg]_{u_1=x_1, u_2=x_1} \right]_{a_i} \\
&\qquad - \frac{1}{2} \mathop{\sum_{g_1+g_2=g}}_{I \sqcup J = S \setminus \{1\}}^{\mathrm{stable}} \mathrm{Sym}_i\left[ x_1^2 \bigg[ \frac{\partial}{\partial x_1} F_{g_1,|I|+1}(z_1, \zz_I) \bigg] \bigg[ \frac{\partial}{\partial x_1} F_{g_2,|J|+1}(\sigma(z_1), \zz_J) \bigg] \right]_{a_i} \\
&\qquad  + \sum_{j=2}^n \mathrm{Sym}_i\Bigg[ \frac{z_j}{z_1 - z_j} \frac{x_1 \frac{\partial}{\partial x_1} F_{g,n-1}(\zz_{S \setminus \{j\}})}{1-z_1s'P'(z_1)} - \frac{z_1}{z_1-z_j} \frac{x_j\frac{\partial}{\partial x_j} F_{g,n-1}(\zz_{S \setminus \{1\}})}{1-z_js'P'(z_j)} \Bigg]_{a_i}
\end{align*}

Now apply the derivative operator $\dd_2 \cdots \dd_n$ to both sides. The third line transforms in the following way, due to Lemma~\ref{lem:w02}, whose statement and proof we postpone in order to declutter the present argument. Note that we do indeed require $\omega_{0,2}$ in the third line, which is not equal to $\Omega_{0,2}$.
\begin{align*}
&\dd_2 \cdots \dd_n~\mathrm{Sym}_i\bigg[ \frac{\partial}{\partial s'} F_{g,n}(z_1, \ldots, z_n) \bigg]_{a_i} = - \frac{1}{2} \, \mathrm{Sym}_i\bigg[ \frac{x_1^2}{(\dd x_1)^2} \Omega_{g-1,n+1}(z_1, \sigma(z_1), \zz_{S \setminus \{1\}}) \bigg]_{a_i} \\
&\qquad- \frac{1}{2} \mathop{\sum_{g_1+g_2=g}}_{I \sqcup J = S \setminus \{1\}}^{\mathrm{stable}} \mathrm{Sym}_i\bigg[ \frac{x_1^2}{(\dd x_1)^2} \Omega_{g_1,|I|+1}(z_1, \zz_I) \, \Omega_{g_2,|J|+1}(\sigma(z_1), \zz_J) \bigg]_{a_i} \\
&\qquad- \sum_{j=2}^n \mathrm{Sym}_i\bigg[ \frac{x_1^2}{(\dd x_1)^2} \omega_{0,2}(z_1, z_j) \, \Omega_{g,n-1}(\zz_{S \setminus \{j\}}) \bigg]_{a_i}
\end{align*}

Use the inductive hypothesis to replace the occurrences of $\Omega_{g',n'}$ on the right side with $\omega_{g',n'}$. Furthermore, gather together the second and third lines and use the symmetry of the situation to obtain the following.
\begin{align*}
&\dd_2 \cdots \dd_n~\mathrm{Sym}_i\bigg[ \frac{\partial}{\partial s'} F_{g,n}(z_1, \ldots, z_n) \bigg]_{a_i} \\
&\qquad= - \Bigg[ \frac{x_1^2}{(\dd x_1)^2} \Bigg[ \omega_{g-1,n+1}(z_1, \sigma(z_1), \zz_{S \setminus \{1\}}) + \mathop{\sum_{g_1+g_2=g}}_{I \sqcup J = S \setminus \{1\}}^\circ \omega_{g_1,|I|+1}(z_1, \zz_I) \, \omega_{g_2,|J|+1}(\sigma(z_1), \zz_J) \Bigg] \Bigg]_{a_i}
\end{align*}

Multiply both sides by $-\frac{\dd x_1}{x_1} \frac{1}{y_1 - \sigma(y_1)}$ and take the principal part at $a_i$ again. Since the expression we are multiplying by is analytic at $a_i$, this just moves the expression inside the principal part.
\begin{align*}
&-\left[ \frac{\dd x_1}{x_1} \frac{1}{(y_1 - \sigma(y_1))} \dd_2 \cdots \dd_n~\mathrm{Sym}_i\bigg[ \frac{\partial}{\partial s'} F_{g,n}(z_1, \ldots, z_n) \bigg] \right]_{a_i} \\
&\qquad= \Bigg[ \frac{x_1}{(y_1 - \sigma(y_1)) \, \dd x_1} \Bigg[ \omega_{g-1,n+1}(z_1, \sigma(z_1), \zz_{S \setminus \{1\}}) + \! \mathop{\sum_{g_1+g_2=g}}_{I \sqcup J = S \setminus \{1\}}^\circ \omega_{g_1,|I|+1}(z_1, \zz_I) \, \omega_{g_2,|J|+1}(\sigma(z_1), \zz_J) \Bigg] \Bigg]_{a_i}
\end{align*}

Now use equation~\eqref{eq:ppart} to evaluate the principal part and recall the definition of the recursion kernel from equation~\eqref{eq:kernel}.
\begin{align*}
&- \left[ \frac{\dd x_1}{x_1} \frac{1}{(y_1 - \sigma(y_1))} \dd_2 \cdots \dd_n~\mathrm{Sym}_i\bigg[ \frac{\partial}{\partial s'} F_{g,n}(z_1, \ldots, z_n) \bigg] \right]_{a_i} \\
&\qquad= \mathop{\text{Res}}_{z=a} K(z_1, z) \Bigg[ \omega_{g-1,n+1}(z, \sigma(z), \zz_{S \setminus \{1\}}) + \mathop{\sum_{g_1+g_2=g}}_{I \sqcup J = S \setminus \{1\}}^\circ \omega_{g_1,|I|+1}(z, \zz_I) \, \omega_{g_2,|J|+1}(\sigma(z), \zz_J) \Bigg]
\end{align*}

Finally, sum over all branch points.
\begin{align*}
&- \sum_{i=1}^d \left[ \frac{\dd x_1}{x_1} \frac{1}{(y_1 - \sigma_i(y_1))} \dd_2 \cdots \dd_n~\mathrm{Sym}_i\bigg[ \frac{\partial}{\partial s'} F_{g,n}(z_1, \ldots, z_n) \bigg] \right]_{a_i} \\
&\qquad= \sum_{i=1}^d \mathop{\text{Res}}_{z=a_i} K_i(z_1, z) \Bigg[ \omega_{g-1,n+1}(z, \sigma_i(z), \zz_{S \setminus \{1\}}) + \mathop{\sum_{g_1+g_2=g}}_{I \sqcup J = S \setminus \{1\}}^\circ \omega_{g_1,|I|+1}(z, \zz_I) \, \omega_{g_2,|J|+1}(\sigma_i(z), \zz_J) \Bigg]
\end{align*}

The left side is simply $\Omega_{g,n}(z_1, \ldots, z_n)$ by Lemma~\ref{lem:omega}, whose statement and proof we postpone in order to declutter the present argument. On the other hand, the right side is simply $\omega_{g,n}(z_1, \ldots, z_n)$ by the definition of the topological recursion. This completes the inductive argument.

All that is left to check are the base cases $\Omega_{0,3} = \omega_{0,3}$ and $\Omega_{1,1} = \omega_{1,1}$. The former statement appears in Subsection~\ref{subsec:lowgenus}. For the latter, we can track through the arguments of the current proof using $(g,n) = (1,1)$ to find that
\[
\left[ \Omega_{1,1}(z_1) \right]_{a_i} = \bigg[ \frac{1}{y_1 - \sigma_i(y_1)} \frac{x_1}{\dd x_1} \omega_{0,2}(z_1, \sigma_i(z_1) ]_{a_i} \bigg]_{a_i}
\]
Now sum over all branch points to obtain $\Omega_{1,1}$ on the left side and the expression for $\omega_{1,1}$ given by the topological recursion on the right side.
\end{proof}

The proof of Theorem~\ref{thm:main} required two lemmas, whose statements and proofs we presently discuss.

\begin{lemma} \label{lem:w02}
If $\mathcal{F}(z)$ satisfies the linear loop equations and $\Omega = \frac{\dd x_1}{x_1} \mathcal{F}(z)$, then
\[
\dd_2~\mathrm{Sym}_i\bigg[ \frac{z_2}{z_1 - z_2} \frac{\mathcal{F}(z_1)}{1-s'z_1P'(z_1)} - \frac{z_1}{z_1 - z_2} \frac{\mathcal{F}(z_2)}{1-s'z_2P'(z_2)} \bigg]_{a_i} = - \, \mathrm{Sym}_i \bigg[ \frac{x_1^2}{(\dd x_1)^2} \omega_{0,2}(z_1, z_2) \, \Omega(\sigma(z_1)) \bigg]_{a_i}.
\]
\end{lemma}

\begin{proof}
The proof is by direct calculation.
\begin{align*}
&\dd_2~\mathrm{Sym}_i\bigg[ \frac{z_2}{z_1 - z_2} \frac{\mathcal{F}(z_1)}{1-s'z_1P'(z_1)} - \frac{z_1}{z_1 - z_2} \frac{\mathcal{F}(z_2)}{1-s'z_2P'(z_2)} \bigg]_{a_i} \\
={}&\dd_2~\mathrm{Sym}_i\bigg[ \frac{z_1}{z_1 - z_2} \frac{\mathcal{F}(z_1)}{1-s'z_1P'(z_1)} - \frac{\mathcal{F}(z_1)}{1-s'z_1P'(z_1)} - \frac{z_2}{z_1 - z_2} \frac{\mathcal{F}(z_2)}{1-s'z_2P'(z_2)} - \frac{\mathcal{F}(z_2)}{1-s'z_2P'(z_2)} \bigg]_{a_i} \\
={}&\dd_2~\mathrm{Sym}_i \bigg[ \frac{z_1}{z_1 - z_2} \frac{\mathcal{F}(z_1)}{1-s'z_1P'(z_1)}  \bigg]_{a_i} \\
={}&- \dd_2~\mathrm{Sym}_i \bigg[ \frac{z_1}{z_1 - z_2} \frac{\mathcal{F}(\sigma(z_1))}{1-s'z_1P'(z_1)}  \bigg]_{a_i} \\
={}&- \mathrm{Sym}_i \bigg[ \frac{\dd z_1 \, \dd z_2}{(z_1 - z_2)^2} \frac{x_1^2}{(\dd x_1)^2} \Omega(\sigma(z_1)) \bigg]_{a_i} \\
={}&- \mathrm{Sym}_i \bigg[ \omega_{0,2}(z_1, z_2) \frac{x_1^2}{(\dd x_1)^2} \Omega(\sigma(z_1)) \bigg]_{a_i}
\end{align*}
The first equality is a straightforward algebraic manipulation. The second equality uses the fact that the second term is annihilated by $\dd_2$ as well as the fact that the third and fourth terms have no pole at $z_1=a_i$. To obtain the third equality, write $\mathcal{G}(z) = \frac{z}{z-z_2} \frac{1}{1-s'zP'(z)}$ and observe that we have the relation $[ (\mathcal{F}(z_1)+\mathcal{F}(\sigma(z_1)) (\mathcal{G}(z_1)+\mathcal{G}(\sigma(z_1)) ]_{a_i} = 0$, since both parentheses are analytic at $z_1=a_i$. The fourth equality evaluates the derivative and uses $\frac{z_1}{1-s'z_1P'(z_1)} = \frac{\dd z_1}{\dd x_1}$. Finally, the fifth inequality simply substitutes the definition of $\omega_{0,2}$.
\end{proof}

\begin{lemma} \label{lem:omega}
Assuming the linear loop equations of Conjecture~\ref{con:poly}, the following is true for $2g-2+n>0$.
\[
\Omega_{g,n}(z_1, \ldots, z_n) = - \dd_2 \cdots \dd_n \sum_{i=1}^d \left[ \frac{\dd x_1}{x_1} \frac{1}{(y_1 - \sigma_i(y_1))} \mathrm{Sym}_i\bigg[ \frac{\partial}{\partial s'} F_{g,n}(z_1, \ldots, z_n) \bigg] \right]_{a_i}
\]
\end{lemma}

\begin{proof}
Given that $F_{g,n}(z_1, \ldots, z_n) \in V(z_1) \otimes V(z_2) \otimes \cdots \otimes V(z_n)$, we deduce that
\begin{align*}
\mathrm{Sym}_i \bigg[ \frac{\partial}{\partial s'} F_{g,n}(z_1, \ldots, z_n) \bigg]_{a_i} &= \mathrm{Sym}_i \bigg[ - y_1 x_1 \frac{\partial}{\partial x_1} F_{g,n}(z_1, \ldots, z_n) \bigg]_{a_i} \\
&= \left[ -(y_1 - \sigma_i(y_1)) x_1 \frac{\partial}{\partial x_1} F_{g,n}(z_1, \ldots, z_n) \right]_{a_i} \qquad \text{for all } i = 1, 2, \ldots, d.
\end{align*}
The first equality holds due to the inclusion $\frac \partial{\partial s} V(z_1) \subset V(z_1)$ discussed in Remark~\ref{rem:commuting} and the second since $x_1 \frac{\partial}{\partial x_1} F_{g,n}(z_1, \ldots, z_n) \in V(z_1) \otimes V(z_2) \otimes \cdots \otimes V(z_n)$. Multiply both sides by $-\frac{\dd x_1}{x_1} \frac{1}{y_1 - \sigma_i(y_1)}$, which is analytic at $z_1 = a_i$, and take the principal part of both sides again to obtain
\[
\left[ \dd_1 F_{g,n}(z_1, \ldots, z_n) \right]_{a_i} = \left[ -\frac{\dd x_1}{x_1} \frac{1}{(y_1 - \sigma_i(y_1))} ~\mathrm{Sym}_i\bigg[ \frac{\partial}{\partial s'} F_{g,n}(z_1, \ldots, z_n) \bigg] \right]_{a_i} \qquad \text{for all } i = 1, 2, \ldots, d.
\]

Since a rational meromorphic form is the sum of its principal parts, summing over $i = 1, 2, \ldots, d$ yields
\[
\dd_1 F_{g,n}(z_1, \ldots, z_n) = - \sum_{i=1}^d \left[ \frac{\dd x_1}{x_1} \frac{1}{(y_1 - \sigma_i(y_1))} ~\mathrm{Sym}_i\left[ \frac{\partial}{\partial s'} F_{g,n}(z_1, \ldots, z_n) \right] \right]_{a_i}.
\]
The desired result then follows by applying $\dd_2 \cdots \dd_n$ to both sides.
\end{proof}

\newpage

\appendix

\section{Table of double Hurwitz numbers} \label{sec:data}
\enlargethispage{4.5pt}

\begin{center}
\begin{tabular}{ccl} \toprule
$g$ & $(\mu_1, \ldots, \mu_n)$ & $DH_{g,n}(\mu_1, \ldots, \mu_n)$ evaluated at $s=1$ \\ \midrule
0 & $(1)$ & $q_1$ \\
0 & $(2)$ & $\frac{1}{2} q_2 + \frac{1}{2} q_1^2$ \\
0 & $(3)$ & $\frac{1}{3} q_3 + q_2 q_1 + \frac{1}{2} q_1^3$ \\
0 & $(4)$ & $\frac{1}{4} q_4 + q_3 q_1 + \frac{1}{2} q_2^2 + 2 q_2 q_1^2 + \frac{2}{3} q_1^4$ \\
0 & $(5)$ & $\frac{1}{5} q_5 + q_4q_1 + q_3q_2 + \frac{5}{2} q_3 q_1^2 + \frac{5}{2} q_2^2 q_1 + \frac{25}{6} q_2 q_1^3 + \frac{25}{24} q_1^5$ \\ \midrule
0 & $(11)$ & $q_2 + \frac{1}{2} q_1^2$ \\
0 & $(21)$ & $q_3 + 2 q_2 q_1 + \frac{2}{3} q_1^3$ \\
0 & $(31)$ & $q_4 + 3 q_3 q_1 + \frac{3}{2} q_2^2 + \frac{9}{2} q_2 q_1^2 + \frac{9}{8} q_1^4$ \\
0 & $(22)$ & $q_4 + 3 q_3 q_1 + q_2^2 + 4 q_2 q_1^2 + q_1^4$ \\
0 & $(41)$ & $q_5 + 4 q_4 q_1 + 4 q_3 q_2 + 8 q_3 q_1^2 + 8 q_2^2 q_1 + \frac{32}{3} q_2 q_1^3 + \frac{32}{15} q_1^5$ \\
0 & $(32)$ & $q_5 + 4 q_4 q_1 + 3 q_3 q_2 + \frac{15}{2} q_3 q_1^2 + 6 q_2^2 q_1 + 9 q_2 q_1^3 + \frac{9}{5} q_1^5$ \\ \midrule
0 & $(111)$ & $3 q_3 + 4 q_2 q_1 + q_1^3$ \\
0 & $(211)$ & $4 q_4 + 9 q_3 q_1 + 4 q_2^2 + 10 q_2 q_1^2 + 2 q_1^4$ \\
0 & $(311)$ & $5 q_5 + 16 q_4 q_1 + 15 q_3 q_2 + \frac{51}{2} q_3 q_1^2 + 24 q_2^2 q_1 + 27 q_2 q_1^3 + \frac{9}{2} q_1^5$ \\
0 & $(221)$ & $5 q_5 + 16 q_4 q_1 + 12 q_3 q_2 + 24 q_3 q_1^2 + 20 q_2^2 q_1 + 24 q_2 q_1^3 + 4 q_1^5$ \\ \midrule
0 & $(1111)$ & $16 q_4 + 27 q_3 q_1 + 12 q_2^2 + 24 q_2 q_1^2 + 4 q_1^4$ \\
0 & $(2111)$ & $25 q_5 + 64 q_4 q_1 + 54 q_3 q_2 + 81 q_3 q_1^2 + 72 q_2^2 q_1 + 70 q_2 q_1^3 + 10 q_1^5$ \\ \midrule
1 & $(2)$ & $\frac{1}{4} q_2 + \frac{1}{12} q_1^2$ \\
1 & $(3)$ & $q_3 + \frac{3}{2} q_2 q_1 + \frac{3}{8} q_1^3$ \\
1 & $(4)$ & $\frac{5}{2} q_4 + 6 q_3 q_1 + \frac{7}{3} q_2^2 + \frac{20}{3} q_2 q_1^2 + \frac{4}{3} q_1^4$ \\
1 & $(5)$ & $5 q_5 + \frac{50}{3} q_4 q_1 + \frac{25}{2} q_3 q_2 + \frac{625}{24} q_3 q_1^2 + \frac{125}{6} q_2^2 q_1 + \frac{625}{24} q_2 q_1^3 + \frac{625}{144} q_1^5$ \\ \midrule
1 & $(11)$ & $\frac{1}{6} q_2 + \frac{1}{24} q_1^2$ \\
1 & $(21)$ & $\frac{3}{2} q_3 + \frac{5}{3} q_2 q_1 + \frac{1}{3} q_1^3$ \\
1 & $(31)$ & $6q_4 + \frac{45}{4} q_3 q_1 + \frac{9}{2} q_2^2 + \frac{81}{8} q_2 q_1^2 + \frac{27}{16} q_1^4$ \\
1 & $(22)$ & $\frac{14}{3} q_4 + 9 q_3 q_1 + \frac{10}{3} q_2^2 + 8 q_2 q_1^2 + \frac{4}{3} q_1^4$ \\
1 & $(41)$ & $\frac{50}{3} q_5 + \frac{136}{3} q_4 q_1 + \frac{104}{3} q_3 q_2 + \frac{176}{3} q_3 q_1^2 + 48 q_2^2 q_1 + \frac{448}{9} q_2 q_1^3 + \frac{64}{9} q_1^5$ \\
1 & $(32)$ & $\frac{25}{2} q_5 + \frac{104}{3} q_4 q_1 + 24 q_3 q_2 + \frac{355}{8} q_3 q_1^2 + 34 q_2^2 q_1 + \frac{147}{4} q_2 q_1^3 + \frac{21}{4} q_1^5$ \\ \midrule
1 & $(111)$ & $\frac{9}{4} q_3 + 2 q_2 q_1 + \frac{1}{3} q_1^3$ \\
1 & $(211)$ & $\frac{40}{3} q_4 + \frac{81}{4} q_3 q_1 + 8 q_2^2 + \frac{91}{6} q_2 q_1^2 + \frac{13}{6} q_1^4$ \\
1 & $(311)$ & $\frac{625}{12} q_5 + \frac{352}{3} q_4 q_1 + \frac{355}{4} q_3 q_2 + \frac{511}{4} q_3 q_1^2 + 104 q_2^2 q_1 + 93 q_2 q_1^3 + \frac{93}{8} q_1^5$ \\
1 & $(221)$ & $\frac{125}{3} q_5 + 96 q_4 q_1 + 68 q_3 q_2 + 104 q_3 q_1^2 + \frac{244}{3} q_2^2 q_1 + \frac{224}{3} q_2 q_1^3 + \frac{28}{3} q_1^5$ \\ \midrule
2 & $(2)$ & $\frac{1}{48} q_2 + \frac{1}{240} q_1^2$ \\
2 & $(3)$ & $\frac{3}{4} q_3 + \frac{27}{40} q_2 q_1 + \frac{9}{80} q_1^3$ \\
2 & $(4)$ & $\frac{41}{6} q_4 + \frac{54}{5} q_3 q_1 + \frac{61}{15} q_2^2 + \frac{364}{45} q_2 q_1^2 + \frac{52}{45} q_1^4$ \\
2 & $(5)$ & $\frac{425}{12} q_5 + \frac{250}{3} q_4 q_1 + \frac{1375}{24} q_3 q_2 + \frac{4375}{48} q_3 q_1^2 + \frac{625}{9} q_2^2 q_1 + \frac{3125}{48} q_2 q_1^3 + \frac{3125}{384} q_1^5$ \\ \midrule
2 & $(11)$ & $\frac{1}{120} q_2 + \frac{1}{720} q_1^2$ \\
2 & $(21)$ & $\frac{27}{40} q_3 + \frac{91}{180} q_2 q_1 + \frac{13}{180} q_1^3$ \\
2 & $(31)$ & $\frac{54}{5} q_4 + \frac{567}{40} q_3 q_1 + \frac{27}{5} q_2^2 + \frac{729}{80} q_2 q_1^2 + \frac{729}{640} q_1^4$ \\
2 & $(22)$ & $\frac{122}{15} q_4 + \frac{54}{5} q_3 q_1 + \frac{182}{45} q_2^2 + \frac{104}{15} q_2 q_1^2 + \frac{13}{15} q_1^4$ \\
2 & $(41)$ & $\frac{250}{3} q_5 + \frac{7448}{45} q_4 q_1 + \frac{1736}{15} q_3 q_2 + \frac{7024}{45} q_3 q_1^2 + \frac{1808}{15} q_2^2 q_1 + \frac{1472}{15} q_2 q_1^3 + \frac{1472}{135} q_1^5$ \\
2 & $(32)$ & $\frac{1375}{24} q_5 + \frac{1736}{15} q_4 q_1 + \frac{1561}{20} q_3 q_2 + \frac{1747}{16} q_3 q_1^2 + \frac{412}{5} q_2^2 q_1 + \frac{2727}{40} q_2 q_1^3 + \frac{303}{40} q_1^5$ \\ \midrule
3 & (2) & $\frac{1}{1440} q_2 + \frac{1}{10080} q_1^2$ \\
3 & (3) & $\frac{9}{40} q_3 + \frac{81}{560} q_2 q_1 + \frac{81}{4480} q_1^3$ \\
3 & (4) & $\frac{73}{9} q_4 + \frac{324}{35} q_3 q_1 + \frac{1094}{315} q_2^2 + \frac{328}{63} q_2 q_1^2 + \frac{328}{567} q_1^4$ \\
3 & (5) & $\frac{8125}{72} q_5 + \frac{15200}{63} q_4 q_1 + \frac{134375}{1008} q_3 q_2 + \frac{1328125}{8064} q_3 q_1^2 + \frac{15625}{126} q_2^2 q_1 + \frac{6640625}{72576} q_2 q_1^3 + \frac{1328125}{145152} q_1^5$ \\ \bottomrule
\end{tabular}
\end{center}

\section{Table of pruned double Hurwitz numbers} \label{sec:pdata}

\begin{center}
\begin{tabular}{ccl} \toprule
$g$ & $(\mu_1, \ldots, \mu_n)$ & $PH_{g,n}(\mu_1, \ldots, \mu_n)$ evaluated at $s=1$ \\ \midrule
0 & $(11)$ & $q_2 + \frac{1}{2} q_1^2$ \\
0 & $(21)$ & $q_3 + q_2 q_1 + \frac{1}{6} q_1^3$ \\
0 & $(31)$ & $q_4 + q_3 q_1 + \frac{1}{2} q_2^2 + \frac{1}{2} q_2 q_1^2 + \frac{1}{24} q_1^4$ \\
0 & $(22)$ & $q_4 + q_3 q_1 + q_2^2 + q_2 q_1^2 + \frac{1}{6} q_1^4$ \\
0 & $(41)$ & $q_5 + q_4 q_1 + q_3 q_2 + \frac{1}{2} q_3 q_1^2 + \frac{1}{2} q_2^2 q_1 + \frac{1}{6} q_2 q_1^3 + \frac{1}{120} q_1^5$ \\
0 & $(32)$ & $q_5 + q_4 q_1 + 2 q_3 q_2 + q_3 q_1^2 + \frac{3}{2} q_2^2 q_1 + \frac{5}{6} q_2 q_1^3 + \frac{11}{120} q_1^5$ \\ \midrule
0 & $(111)$ & $3 q_3 + 4 q_2 q_1 + q_1^3$ \\
0 & $(211)$ & $4 q_4 + 6 q_3 q_1 + 4 q_2^2 + 6 q_2 q_1^2 + q_1^4$ \\
0 & $(311)$ & $5 q_5 + 8 q_4 q_1 + 12 q_3 q_2 + 9 q_3 q_1^2 + 12 q_2^2 q_1 + 8 q_2 q_1^3 + q_1^5$ \\
0 & $(221)$ & $5 q_5 + 8 q_4 q_1 + 12 q_3 q_2 + 9 q_3 q_1^2 + 12 q_2^2 q_1 + 8 q_2 q_1^3 + q_1^5$ \\ \midrule
0 & $(1111)$ & $16 q_4 + 27 q_3 q_1 + 12 q_2^2 + 24 q_2 q_1^2 + 4 q_1^4$ \\
0 & $(2111)$ & $25 q_5 + 48 q_4 q_1 + 54 q_3 q_2 + 54 q_3 q_1^2 + 60 q_2^2 q_1 + 46 q_2 q_1^3 + 6 q_1^5$ \\ \midrule
1 & $(2)$ & $\frac{1}{4} q_2 + \frac{1}{12} q_1^2$ \\
1 & $(3)$ & $q_3 + q_2 q_1 + \frac{5}{24} q_1^3$ \\
1 & $(4)$ & $\frac{5}{2} q_4 + 3 q_3 q_1 + \frac{11}{6} q_2^2 + \frac{5}{2} q_2 q_1^2 + \frac{3}{8} q_1^4$ \\
1 & $(5)$ & $5 q_5 + \frac{20}{3} q_4 q_1 + 9 q_3 q_2 + \frac{51}{8} q_3 q_1^2 + 8 q_2^2 q_1 + 5 q_2 q_1^3 + \frac{7}{12} q_1^5$ \\ \midrule
1 & $(11)$ & $\frac{1}{6} q_2 + \frac{1}{24} q_1^2$ \\
1 & $(21)$ & $\frac{3}{2} q_3 + \frac{3}{2} q_2 q_1 + \frac{7}{24} q_1^3$ \\
1 & $(31)$ & $6q_4 + \frac{33}{4} q_3 q_1 + \frac{13}{3} q_2^2 + \frac{41}{6} q_2 q_1^2 + \frac{25}{24} q_1^4$ \\
1 & $(22)$ & $\frac{14}{3} q_4 + 6 q_3 q_1 + \frac{10}{3} q_2^2 + \frac{29}{6} q_2 q_1^2 + \frac{17}{24} q_1^4$ \\
1 & $(41)$ & $\frac{50}{3} q_5 + \frac{82}{3} q_4 q_1 + \frac{63}{2} q_3 q_2 + \frac{223}{8} q_3 q_1^2 + \frac{94}{3} q_2^2 q_1 + \frac{265}{12} q_2 q_1^3 + \frac{65}{24} q_1^5$ \\
1 & $(32)$ & $\frac{25}{2} q_5 + \frac{58}{3} q_4 q_1 + \frac{45}{2} q_3 q_2 + \frac{151}{8} q_3 q_1^2 + \frac{64}{3} q_2^2 q_1 + \frac{173}{12} q_2 q_1^3 + \frac{41}{24} q_1^5$ \\ \midrule
1 & $(111)$ & $\frac{9}{4} q_3 + 2 q_2 q_1 + \frac{1}{3} q_1^3$ \\
1 & $(211)$ & $\frac{40}{3} q_4 + 18 q_3 q_1 + 8 q_2^2 + \frac{79}{6} q_2 q_1^2 + \frac{11}{6} q_1^4$ \\
1 & $(311)$ & $\frac{625}{12} q_5 + \frac{272}{3} q_4 q_1 + \frac{173}{2} q_3 q_2 + \frac{707}{8} q_3 q_1^2 + 86 q_2^2 q_1 + \frac{190}{3} q_2 q_1^3 + \frac{179}{24} q_1^5$ \\
1 & $(221)$ & $\frac{125}{3} q_5 + \frac{208}{3} q_4 q_1 + 68 q_3 q_2 + \frac{263}{4} q_3 q_1^2 + \frac{196}{3} q_2^2 q_1 + \frac{139}{3} q_2 q_1^3 + \frac{16}{3} q_1^5$ \\ \midrule
2 & $(2)$ & $\frac{1}{48} q_2 + \frac{1}{240} q_1^2$ \\
2 & $(3)$ & $\frac{3}{4} q_3 + \frac{19}{30} q_2 q_1 + \frac{5}{48} q_1^3$ \\
2 & $(4)$ & $\frac{41}{6} q_4 + \frac{171}{20} q_3 q_1 + \frac{161}{40} q_2^2 + \frac{439}{72} q_2 q_1^2 + \frac{119}{144} q_1^4$ \\
2 & $(5)$ & $\frac{425}{12} q_5 + 56 q_4 q_1 + 55 q_3 q_2 + \frac{821}{16} q_3 q_1^2 + \frac{3689}{72} q_2^2 q_1 + \frac{567}{16} q_2 q_1^3 + \frac{4627}{1152} q_1^5$ \\ \midrule
2 & $(11)$ & $\frac{1}{120} q_2 + \frac{1}{720} q_1^2$ \\
2 & $(21)$ & $\frac{27}{40} q_3 + \frac{179}{360} q_2 q_1 + \frac{17}{240} q_1^3$ \\
2 & $(31)$ & $\frac{54}{5} q_4 + \frac{513}{40} q_3 q_1 + \frac{647}{120} q_2^2 + \frac{389}{48} q_2 q_1^2 + \frac{637}{640} q_1^4$ \\
2 & $(22)$ & $\frac{122}{15} q_4 + \frac{189}{20} q_3 q_1 + \frac{182}{45} q_2^2 + \frac{427}{72} q_2 q_1^2 + \frac{521}{720} q_1^4$ \\
2 & $(41)$ & $\frac{250}{3} q_5 + \frac{1198}{9} q_4 q_1 + \frac{915}{8} q_3 q_2 + \frac{9193}{80} q_3 q_1^2 + \frac{37199}{360} q_2^2 q_1 + \frac{5733}{80} q_2 q_1^3 + \frac{14651}{1920} q_1^5$ \\
2 & $(32)$ & $\frac{1375}{24} q_5 + \frac{266}{3} q_4 q_1 + \frac{619}{8} q_3 q_2 + \frac{751}{10} q_3 q_1^2 + \frac{24629}{360} q_2^2 q_1 + \frac{33397}{720} q_2 q_1^3 + \frac{28123}{5760} q_1^5$ \\ \midrule
3 & (2) & $\frac{1}{1440} q_2 + \frac{1}{10080} q_1^2$ \\
3 & (3) & $\frac{9}{40} q_3 + \frac{361}{2520} q_2 q_1 + \frac{103}{5760} q_1^3$ \\
3 & (4) & $\frac{73}{9} q_4 + \frac{2403}{280} q_3 q_1 + \frac{17497}{5040} q_2^2 + \frac{3437}{720} q_2 q_1^2 + \frac{27187}{51840} q_1^4$ \\
3 & (5) & $\frac{8125}{72} q_5 + \frac{10456}{63} q_4 q_1 + \frac{37137}{280} q_3 q_2 + \frac{247069}{1920} q_3 q_1^2 + \frac{78973}{720} q_2^2 q_1 + \frac{1847309}{25920} q_2 q_1^3 + \frac{4427}{640} q_1^5$ \\
\bottomrule
\end{tabular}
\end{center}

\begin{small}
\bibliographystyle{habbrv}
\bibliography{double-hurwitz-topological-recursion.bib}

\textsc{School of Mathematical Sciences, Monash University, VIC 3800, Australia} \\
\emph{Email:} \href{mailto:norm.do@monash.edu}{norm.do@monash.edu}

\textsc{St. Petersburg Department of the Steklov Mathematical Institute, Fontanka 27, St. Petersburg 191023, Russia} \\
\emph{Email:} \href{mailto:max.karev@gmail.com}{max.karev@gmail.com}

\end{small}

\end{document}